\documentclass[preprint,authoryear,12pt]{article}

\usepackage{pdflscape}
\usepackage{geometry}
\usepackage{latexsym}
\usepackage{indentfirst}
\usepackage{graphicx}
\usepackage{pdfpages}
\usepackage{amsmath}
\usepackage{amssymb}
\usepackage{amsthm}
\usepackage{bm}
\newtheorem{prop}{Proposition}[section]
\newtheorem*{defi*}{Definition}

\newtheorem{lem}[prop]{Lemma}
\newtheorem{rem}[prop]{Remark}
\newtheorem{thm}[prop]{Theorem}
\newtheorem{coro}[prop]{Corollary}

\begin{document}

\noindent
{\Large\bf  A new approach to the existence of invariant measures for Markovian semigroups}\\

\noindent
Lucian Beznea\footnote{Simion Stoilow Institute of Mathematics  of the Romanian Academy,
Research unit No. 2, P.O. Box  1-764, RO-014700 Bucharest, Romania,
and University of Bucharest, Faculty of Mathematics and Computer Science
(e-mail: lucian.beznea@imar.ro)},   
Iulian C\^{i}mpean\footnote{Simion Stoilow Institute of Mathematics  of the Romanian Academy,
Research unit No. 2, P.O. Box  1-764, RO-014700 Bucharest, Romania,
(e-mail: iulian.cimpean@imar.ro)},   and 
Michael R\"ockner\footnote{Fakult\"at f\"ur Mathematik, Universit\"at Bielefeld,
Postfach 100 131, D-33501 Bielefeld, Germany.
(e-mail: roeckner@mathematik.uni-bielefeld.de)}
\\[5mm]

\noindent
{\bf Abstract.}
We give a new, two-step approach to prove existence of finite invariant measures for a given Markovian semigroup.
First, we identify a convenient {auxiliary measure} and then we prove conditions equivalent to the existence of an invariant finite measure which is absolutely continuous with respect to it.
As applications, we give a short proof for the result of Lasota and Szarek on invariant measures and we obtain a unifying generalization of different versions for Harris' ergodic theorem which provides an answer to an open question of Tweedie. 
We show that for a nonlinear SPDE on a Gelfand triple, the strict coercivity condition is sufficient to guarantee the existence of a unique invariant probability measure for the associated semigroup, once it satisfies a Harnack type inequality. 
A corollary of the main result shows that any uniformly bounded semigroup on
 $L^p$ possesses an invariant measures and we give some applications to sectorial perturbations of  Dirichlet forms. \\[2mm]



%60J35  	%Transition functions, generators and resolvents
%60J40  	%Right processes
%60J45  	%Probabilistic potential theory
%31C25   %Dirichlet spaces
%47D07  	%Markov semigroups and applications to diffusion processes 
%47D03  	%Groups and semigroups of linear operators 
%47A35  	%Ergodic theory
%37A30  %Ergodic theorems, spectral theory, Markov operators
%37C40 %Smooth ergodic theory, invariant measures
%37L40 %Invariant measures (Infinite-dimensional dissipative dynamical systems)

%60J25  	Continuous-time Markov processes on general state spaces
%82B10  	Quantum equilibrium statistical mechanics (general)

\section{Introduction}

The invariant measure is a key concept in ergodic theory.
In this paper we deal with the question of existence of finite invariant measures for Markovian semigroups.
This problem has been studied by many authors over the last decades, from various points of view; see e.g. the monographs 
\cite{MeTw93a}, \cite{DaZa96}, and the references therein.

If the underlying space $E$ is a Polish space, the semigroup is given by the transition probabilities of a Markov process and is Feller 
(i.e., it maps the space of bounded continuous real-valued  functions on $E$ into itself), then one can obtain the existence of an invariant measure by applying the result of \cite{LaSz06}, provided that there is a compact subset of $E$ which is infinitely often visited by the process.
Although these hypotheses are verified in many examples, sometimes they are quite difficult or even impossible to check, especially if the state space is of infinite dimensions.
Another technique to obtain invariant measures is to make use of Harris' theorem and its refined versions, 
cf. e.g. \cite{MeTw93a}, \cite{MeTw93b}, \cite{MeTw93c}, \cite{MeTw93d}, \cite{DoFoGu09}, and \cite{Ha10}.
In contrast to the previously mentioned, these results involve non-topological assumptions such as the existence of {\it small} sets (in the sense which will be made precise in Subsection 3.2 below) that are infinitely often visited.
This kind of test sets are encountered, provided the associated process is irreducible; see \cite{MeTw93a}, Theorem 5.2.2.
Invariant measures have also been investigated from an analytic perspective, as in \cite{BoRoZh00} and \cite{Hi00}, by working with strongly continuous Markovian semigroups on $L^p$, $1<p<\infty$.
Examples of this situation arise by considering sectorial perturbations of Dirichlet forms satisfying some functional inequalities (see Subsection 3.3 below).

The purpose of this paper is to give a new approach to the existence of invariant measures for Markovian semigroups, 
consisting of two steps. 
First, we construct a convenient {\it auxiliary} measure $m$ (see Proposition \ref{prop 2.7}) and then we give conditions on the pair $(P_t, m)$ 
which characterize the existence of a non-zero integrable co-excessive function for $(P_t)_{t \geq 0}$, 
regarded as a semigroup on $L^\infty (m)$, which is equivalent to the existence of an invariant measure for $(P_t)_{t\geq 0}$, which is absolutely continuous with respect to $m$ (see Theorem \ref{thm 2.3} below and also Theorems \ref{thm 2.5}, \ref{thm 2.3.1} as its useful variants).
Therefore, we call the procedure proposed above {\it the two-step approach}; see Subsection 2.2.
We point out that our main results are entirely measure theoretic and also do not involve irreducibility properties of the semigroup.

Several applications are considered: In Subsection 3.1, although not in its full generality, we give a short proof of the well known result of Lasota and Szarek \cite{LaSz06}. 
Here, the two-step approach gives an additional benefit because it implies the absolute continuity of the obtained invariant measure with respect to the auxiliary measure.

In Subsection 3.2 we unify various versions of Harris' ergodic theorem to a more general one; see Theorem \ref{thm 3.8.0}, which contains all of these as special cases. 
As a byproduct, in Corollary \ref{coro 3.10} we give an answer to an open question mentioned by Tweedie \cite{Tw01}. 

In Subsection 3.3 we show that for a nonlinear SPDE on a Gelfand triple $V \subset H \subset V^\ast$, under a Wang's Harnack type inequality, the strict coercivity condition with respect to the $H$-norm is sufficient to guarantee the existence of a unique invariant probability measure for the solution; see Proposition \ref{prop 3.12}.
This result improves the ones from \cite{Li09} and \cite{Wa13} where the embedding $V \subset H$ must be compact and the strict coercivity is considered with respect to the stronger $V$-norm.
We also consider a perturbation of a Markov kernel satisfying a combined Harnack-Lyapunov condition, for which the result of Tweedie (Theorem \ref{thm 3.5} below) can not be used, but for which our two-step approach works easily. 
We also discuss the applicability of Harris' result to this kind of perturbation. 
The  last part of this subsection was written taking into account a kind remark of Martin Hairer, which lead to the statement of Proposition \ref{prop 3.16}.

In Subsection 3.4 we study the case of uniformly bounded $C_0$-semigroups on $L^p$, $p \geq 1$. Implementing our two-step approach we obtain new applications for semigroups coming from small perturbations of Dirichlet forms, generalizing \cite{BoRoZh00} and \cite{Hi98}.

\section{Existence of invariant measures}

\subsection{Preliminaries}
Throughout, $(E, \mathcal{B})$ is a measurable space and $m$ a finite positive measure on it. 
Let $L^p(m)$, $1 \leq p \leq \infty$ be the standard Lebesgue spaces and $\| \cdot \|_p$ the associated norms.

We denote by $L^p_+(m)$ the space of positive elements from $L^p(m)$. 
A linear operator $T$ on $L^p(m)$ is called positivity preserving if $T(L^p_+(m)) \subset L^p_+(m)$. 
Note that (as in \cite{JaSc06}, Lemma 1.2), any positivity preserving operator on $L^p(m)$, $1 \leq p \leq \infty$ is automatically bounded. 
$T$ is called sub-Markovian (resp. Markovian) if it is positivity preserving and $T1 \leq 1$ (resp. $T1 = 1$). 
If $T$ is a sub-Markovian operator on $L^p(m)$ for some $p \geq 1$, then $T$ extends to a sub-Markovian operator on $L^r(m)$ for all $p\leq r \leq \infty$; see [Ja Sch 06]. 
Moreover, if $(E, \mathcal{B})$ is a Lusin measurable space then $T$ is given by a sub-Markovian kernel on $(E, \mathcal{B})$; cf. e.g. \cite{BeBo04}, Lemma A.1.9.

We recall that a {\it transition function} on $(E,\mathcal{B})$ is a family $(P_t)_{t\geq 0}$ of sub-Markovian kernels on $(E,\mathcal{B})$ such that $P_t(P_s f) = P_{s+t}f$ for all positive $\mathcal{B}$-measurable functions $f$ and all $s, t \in \mathbb{R}_+$.
The transition function $(P_t)_{t\geq 0}$ is called Markovian provided that for all $t$ (or for only one $t > 0$) the kernel $P_t$ is Markovian. 
The transition function $(P_t)_{t \geq 0}$ is called {\it measurable} if the function $(t, x) \mapsto P_tf(x)$ is $\mathcal{B}([0,\infty))\otimes \mathcal{B}$-measurable for all positive $\mathcal{B}$-measurable functions $f$.

Hereinafter, $(P_t)_{t \geq 0}$ and $m$ are satisfying either

(A$_1$) $(P_t)_{t \geq 0}$ is a strongly continuous semigroup of Markovian operators on $L^p(m)$ for some $p \geq 1$,

or

(A$_2$) $(P_t)_{t\geq 0}$ is a measurable Markovian transition function on $(E, \mathcal{B})$ such that $m(f) = 0 \Rightarrow m(P_tf) = 0$ for all $t > 0$ and all positive $\mathcal{B}$-measurable functions $f$. In this case, we say that $m$ in an {\it auxiliary} measure for $(P_t)_{t\geq 0}$.

Our goal is to investigate the existence of a nonzero invariant measure $\nu$ for $(P_t)_{t \geq 0}$, i.e. a nonzero finite positive measure $\nu$
on $(E, \mathcal{B})$ such that $\int P_tf d\nu = \int f d\nu$ for all $t > 0$ and all bounded $\mathcal{B}$-measurable functions $f$.

As a matter of fact, the class of invariant measures to be studied consists of absolutely continuous measures with respect to the fixed measure $m$, whose densities are invariant functions for the dual semigroup. 
Inspired by well known ergodic properties for semigroups and resolvents (see for example \cite{BeCiRo15}), our main idea in order to produce co-invariant functions is to apply some compactness results in $L^1(m)$, not for $(P_t)_{t \geq 0}$ but for its adjoint semigroup.
However, if $(P_t)_{t \geq 0}$ satisfies (A$_1$) or (A$_2$), it is not obvious that its adjoint semigroup may be regarded as a semigroup acting on $L^1(m)$. 
Apparently, another difficulty when $(P_t)_{t > 0}$ satisfies (A$_2$) is the lack of Bochner integrability of its adjoint, on $(L^{\infty}(m))^{\ast}$. 
All these issues are clarified by the following result, whose proof is presented in Appendix.

\vspace{0.2cm}

\begin{lem} \label{lem 2.1}
i) Assume that $(P_t)_{t \geq 0}$ satisfies (A$_1$) for $p > 1$. Then the adjoint semigroup $(P_t^{\ast})_{t \geq 0}$ on $L^{p'}(m)$, $\dfrac{1}{p} + \dfrac{1}{p'} = 1$, may be regarded as a $C_0$-semigroup of positivity preserving operators on $L^1(m)$.

ii) Assume that $(P_t)_{t \geq 0}$ satisfies either (A$_2$) or (A$_1$) for $p=1$. Then the adjoint semigroup $(P_t^{\ast})_{t \geq 0}$ on $(L^{\infty}(m))^{\ast}$ may be regarded as a semigroup of positivity preserving operators acting on $L^1(m)$, and there exists $(\varphi_t)_{t \geq 0} \subset L_+^1(m)$ with the following properties:

ii.1) $m(\int_{0}^{t}P_sf ds) = m(f \varphi_t)$ \quad for all $t \geq 0$ and $f \in L^\infty(m)$.

ii.2) $P_t^{\ast}\varphi_s = \varphi_{t + s} - \varphi_s$ \quad for all $s, t \geq 0$

ii.3) $\| \dfrac{1}{s} (\varphi_{t + s} - \varphi_s) \|_{L^1} \mathop{\longrightarrow}\limits_{s \to \infty} 0$ for all $t \geq 0$.

\end{lem}

\begin{rem} \label{rem 2.2}

If $(P_t)_{t \geq 0}$ satisfies (A$_1$) for $p>1$, then by Lemma \ref{lem 2.1}, i), the Bochner integrals $\widetilde{\varphi}_t := \int_{0}^{t} P_s^{\ast}1 ds$ are well defined in $L^1(m)$ for all $t > 0$, and $(\widetilde{\varphi}_t)_{t > 0}$ satisfies ii.1) - ii.3). 
On the other hand, if $(P_t)_{t \geq 0}$ is either as in (A$_2$) or as in (A$_1$) for $p=1$, then $ t \mapsto P_t^{\ast}1$ may no longer be integrable on compact intervals. 
From this point of view, $(\varphi_t)_{t > 0}$ in Lemma \ref{lem 2.1}, ii) should be regarded as a substitute for $(\int_{0}^{t} P_s^{\ast}1 ds)_{t > 0}$.

\end{rem}

Recall that if $(P_t)_{t \geq 0}$ is a measurable Markovian transition function on $(E, \mathcal{B})$ (or satisfies (A$_1$)), then the corresponding resolvent $(R_{\alpha})_{\alpha > 0}$ is defined by
$$
R_{\alpha} f(x) = \int_{0}^{\infty} e^{-\alpha t} P_t f(x) dt \; 
$$
for all bounded $\mathcal{B}$-measurable functions $f$, ($m$-a.e.) $x \in E$, and $\alpha > 0$. 

The following known result shows that the problem of existence of invariant measures for a semigroup of operators may be stated in terms of a single operator. 

\begin{prop} \label{prop 2.3} %proposition 2.3
The following assertions hold for a measurable Markovian transition function $(P_t)_{t \geq 0}$ on $(E, \mathcal{B})$.

i) The measure $m$ is invariant for $(P_t)_{t \geq 0}$ if and only if $m \circ \alpha R_\alpha = m$ for some (hence for all) $\alpha >0$.

ii) $(P_t)_{t \geq 0}$ possesses an invariant measure if and only if there exists $t_0 > 0$ such that $P_{t_0}$ possesses an invariant measure.

\end{prop} 

\begin{proof}
i). Clearly, if $m$ is $(P_t)_{t \geq 0}$-invariant then $m \circ \alpha R_{\alpha} = m$ for all $\alpha > 0$, by the definition of the resolvent. 
Conversely, if $m \circ \alpha R_{\alpha} = m$, because $R_\alpha P_tf - R_\alpha f = \alpha R_\alpha (\int_0^t P_s f ds) - \int_0^t P_s f ds$ for all bounded and $\mathcal{B}$-measurable functions $f$, it follows that $m$ is  $(P_t)_{t \geq 0}$-invariant. 

ii). If $m$ is $P_{t_0}$-invariant, then one can easily check that $\frac{1}{t_0}\int_0^{t_0} m \circ P_s ds$ is $(P_t)_{t \geq 0}$-invariant.
\end{proof}

\subsection{The main results}

Let $(P_t)_{t \geq 0}$ and $m$ be as in (A$_1$) or (A$_2$). For a sequence $(t_n)_n \nearrow \infty$ we define the index $c((P_t)_t, m, (t_n)_n)$ by
$$
c((P_t)_t, m, (t_n)_n) := \mathop{\lim}\limits_{\varepsilon \searrow 0} \mathop{\sup}\limits_{\mathop{A \in \mathcal{B}}\limits_{m(A) \leq \varepsilon}} \mathop{\sup}\limits_{n} \dfrac{1}{t_n} \int_{0}^{t_n} m(P_s1_A) \; ds.
$$

Note that $c((P_t)_t, m, (t_n)_n) = 0$ if and only if either $(\dfrac{1}{t_n}\int_{0}^{t_n}P_s^{\ast}1 ds)_{n \geq 1}$ or $(\dfrac{1}{t_n}\varphi_{t_n})_{n \geq 1}$ (according to which of the assumptions $(A_1)$ or $(A_2)$ is satisfied) is uniformly integrable, or equivalently, by Dunfurd-Pettis theorem, it is weakly relatively compact in $L^1(m)$.
From this point of view, $c((P_t)_t, m, (t_n)_n)$ can be regarded as a measurement for the non-uniformly integrability of $(\dfrac{1}{t_n}\int_{0}^{t_n}P_s^{\ast}1 ds)_{n \geq 1}$, resp. $(\dfrac{1}{t_n}\varphi_{t_n})_{n \geq 1}$ in $L^1(m)$.
On the other hand, $c((P_t)_t, m, (t_n)_n)$ can also be interpreted as an index of non-uniformly absolute continuity of the Krylov-Bogoliubov measures $(\dfrac{1}{t_n}\int_{0}^{t_n}m \circ P_s ds)_n$ with respect to $m$.

We say that a positive finite measure $m$ is {\it almost} invariant for $(P_t)_{t \geq 0}$ if $(A_1)$ or $(A_2)$ are satisfied w.r.t. $m$ and there exist $\delta \in [0, 1)$ and a set function $\phi : \mathcal{B} \rightarrow \mathbb{R}_+$ which is absolutely continuous with respect to $m$ (i.e. $\mathop{\lim}\limits_{m(A) \to 0} \phi(A) = 0$) such that 

$$m (P_t1_A) \leq \delta m(E) + \phi(A) \quad \mbox{for all} \; t > 0. \leqno{(2.1)}$$
Analogously, $m$ is said to be {\it mean} almost invariant (w.r.t. $(t_n)_n \nearrow \infty$) if there exist $\delta$ and $\phi$ as above such that

$$\dfrac{1}{t_n}\int_{0}^{t_n}m(P_t (1_A)) dt \leq \delta m(E) + \phi(A) \quad \mbox{for all} \; n. \leqno{(2.2)}$$
Clearly, for a positive finite measure we have the following implications between the above three properties:

\begin{center}
invariant $\Rightarrow$   almost invariant $\Rightarrow$   mean almost invariant
\end{center}

We are now in the position to present our main result.

\begin{thm} \label{thm 2.3}

Assume that $(P_t)_{t \geq 0}$ and $m$ are as in (A$_1$) or (A$_2$). 
The following assertions are equivalent.

i) There exists a nonzero positive finite invariant measure for $(P_t)_{t \geq 0}$ which is absolutely continuous with respect to $m$.

\vspace{0.2cm}

ii) $m$ is almost invariant.

\vspace{0.2cm}

iii) $m$ is mean almost invariant with respect to every $(t_n)_{n} \nearrow \infty$.

\vspace{0.2cm}
iv) For all sequences $(t_n)_{n} \nearrow \infty$ it holds that
$$
c((P_t)_t, m, (t_n)_n) < m(E). \leqno(2.3)
$$

\vspace{0.2cm}

v) There exists a sequence $(t_n)_{n}$ of positive real numbers increasing to infinity for which condition (2.3) is satisfied.

\end{thm}

\begin{proof}
i) $\Rightarrow$ ii). 
Let $0 \leq \rho \in L^1(m)$ such that the measure $\rho \cdot m$ is nonzero and $(P_t)_{t \geq 0}$-invariant. 
Set $\gamma := m((1-\rho)^+)m(E)^{-1}$ and note that since $\rho \cdot m$ is nonzero it follows that $\gamma \in [0, 1)$.
Also, let $c \geq 0$ be such that $m(\rho 1_{[\rho > c]}) \leq \dfrac{1 - \gamma}{2} m(E)$. 
Then, for $A \in \mathcal{B}$ and $t \geq 0$ we have that $m(P_t 1_A) = m(\rho P_t 1_A) + m ((1-\rho) P_t 1_A) = m(\rho 1_A) + m ((1-\rho) P_t 1_A) \leq m(\rho 1_{A \cap [\rho \leq c]}) + m(\rho 1_{A \cap [\rho > c]}) + m((1-\rho)^+ P_t 1_A) \leq c m(A) + m(\rho 1_{[\rho > c]}) + m((1-\rho)^+) \leq c m(A) + \dfrac{1 - \gamma}{2} m(E) + \gamma m(E) = c m(A) +\dfrac{1+ \gamma}{2}m(E)$.
Therefore, we obtained that $m$ is almost invariant with $\phi(A)=cm(A)$ and $\delta = \dfrac{1+ \gamma}{2}$.

The implications ii) $\Rightarrow$ iii) and iv) $\Rightarrow$ v) are clear.

iii) $\Rightarrow$ iv). 
Let $(t_n)_n \nearrow \infty$, $\delta \in [0, 1)$, and a function $\phi : \mathcal{B} \rightarrow \mathbb{R}_+$ which is absolutely continuous with respect to $m$ such that (2.2) holds.
Then 
$$c((P_t)_t, m, (t_n)_n) \leq \mathop{\lim}\limits_{\varepsilon \searrow 0} \mathop{\sup}\limits_{\mathop{A \in \mathcal{B}}\limits_{m(A) \leq \varepsilon}} (\delta m(E) + \phi(A))) = \delta m(E) < m(E).$$
Therefore, iv) is satisfied.

v) $\Rightarrow$ i). 
Assume (A$_1$). Let $(P_t^{\ast})_{t > 0}$ be as in Lemma \ref{lem 2.1}, i) and define $f_n := \dfrac{1}{t_n} \int_{0}^{t_n} P_s^{\ast}1 ds$. 
Then $(f_n)_n \subset L_+^1(m)$ and is $L^1$-bounded since $\int f_n dm = \dfrac{1}{t_n} \int_{0}^{t_n} \int P_s1 dm ds = m(E)$.

By Chacon's Biting lemma (see Appendix, A.2), there exist a subsequence $(f_{n_k})_{n \geq 1}$, $f \in L^1(m)$, and a decreasing sequence of "bits" $(B_r)_{r \geq 1} \subset \mathcal{B}$ such that $m(B_r) \mathop{\longrightarrow}\limits_{r} 0$ and for all $r \geq 1$ the sequence $(1_{B_r^c}f_{n_k})_{k \geq 1}$ is weakly convergent to $1_{B_r^c}f$. 
On the other hand, by Koml\'os lemma (see Appendix, A.3), there exists a subsequence $(g_k)_{k \geq 1}$ of $(f_{n_k})_{n \geq 1}$ such that $\dfrac{g_1 + \ldots + g_k}{k}$ is $m$-a.e. convergent to some $g \in L^1(m)$.

Without loss, we may assume that $s_k = \dfrac{f_{n_1} + \ldots + f_{n_k}}{k}$ converges $m$-a.e. to $g$. 
One can easily check that $g = f$\; $m$-a.e.; see for example [Fl99], Proposition 3.

We claim that $P_t^{\ast} f \leq f$ $m$-a.e. for al $t > 0$. To see this, first note that
$$
\int \int_{t_{n_i}}^{t_{n_i}+t} P_r^{\ast}1 dr dm = \int_{t_{n_i}}^{t_{n_i}+t}\int P_r1 dm dr = t m(E) = \int\int_{0}^{t} P_r^{\ast}1 dr dm, 
$$
{hence}
$$
(h_i)_{i \geq 1} := (\dfrac{1}{t_{n_i}}\int_{t_{n_i}}^{t_{n_i}+t} P_r^{\ast}1 dr)_{i \geq 1} \; {\rm and} \; (g_i)_{i \geq 1} := (\dfrac{1}{t_{n_i}} \int_{0}^{t}P_r^{\ast}1 dr)_{i \geq 1}
$$

\noindent are both convergent to 0 in $L^1(m)$. 
By passing to a subsequence, without loss of generality we may assume that $(h_i)_{i \geq 1}$ and $(g_i)_{i \geq 1}$ converge to 0 $m$-a.e.
Then 
$$
P_t^{\ast}f = P_t^{\ast}(\mathop{\lim}\limits_{k} s_k) = P_t^{\ast}(\mathop{\sup}\limits_{N}\mathop{\inf}\limits_{k \geq N} s_k) = \mathop{\sup}\limits_{N}P_t^{\ast}(\mathop{\inf}\limits_{k \geq N} s_k)
$$
$$
\leq \mathop{\sup}\limits_{N}\mathop{\inf}\limits_{k \geq N} P_t^{\ast} (s_k) =  \mathop{\lim\inf}\limits_{k} P_t^{\ast} (s_k) = \mathop{\lim\inf}\limits_{k} \dfrac{1}{k} \mathop{\sum}\limits_{i = 1}^{k} \dfrac{1}{t_{n_i}} \int_{0}^{t_{n_i}} P_{t + s}^{\ast}1 ds
$$
$$
= \mathop{\lim\inf}\limits_{k} \dfrac{1}{k} \mathop{\sum}\limits_{i = 1}^{k} \dfrac{1}{t_{n_i}}(\int_{0}^{t_{n_i}} P_s^{\ast}1 ds \; + \int_{t_{n_i}}^{t_{n_i} + t} P_s^{\ast}1 ds \; - \int_{0}^{t} P_s^{\ast}1 ds) = \mathop{\lim}\limits_{k} s_k = f \; m-{\rm a.e.}.
$$

If we set $\nu = f \cdot m$ then $\int P_t g d\nu = \int g P_t^{\ast} f dm \leq \int g d\nu$, hence $\nu$ is sub-invariant. 
Since $P_t 1 = 1,\; t > 0$, it follows that $\nu$ is invariant. 
However, we still have to check that $\nu$ is non-zero. 
Indeed, by \cite{Fl99} we have that
$$
\mathop{\lim}\limits_{r} \mathop{\lim\sup}\limits_{k}\int_{B_r} f_{n_k} dm = \mathop{\inf}\limits_{\varepsilon > 0} \mathop{\sup}\limits_{m(A) < \varepsilon} \mathop{\sup}\limits_{k} \int_{A} f_{n_k} dm 
$$
$$
\leq \mathop{\inf}\limits_{\varepsilon > 0} \mathop{\sup}\limits_{m(A) < \varepsilon} \mathop{\sup}\limits_{n} \int_{A} f_n dm = c((P_t)_t, m, (t_n)_n) < m(E).
$$

Therefore, $\nu(E) = \int f dm = \mathop{\sup}\limits_{r} \int_{B_r^c} f dm = \mathop{\lim}\limits_{r} \mathop{\lim}\limits_{k} \int_{B_r^c} f_{n_k} dm$
$$
= \mathop{\lim}\limits_{r} \mathop{\lim\inf}\limits_{k} (\int f_{n_k} dm - \int_{B_r} f_{n_k} dm) 
%$$
%$$
\geq m(E) - \mathop{\lim}\limits_{r} \mathop{\lim\sup}\limits_{k} \int_{B_r} f_{n_k} dm > 0.
$$
Finally, if $(P_t)_{t \geq 0}$ is as in (A$_2$), then the proof follows the same lines as above once we replace $(\int_{0}^{t} P_s^{\ast}1 ds)_{t > 0}$ by $(\varphi_t)_{t \geq 0}$ given by Lemma \ref{lem 2.1}, ii); see also Remark \ref{rem 2.2}. 
\end{proof}

\begin{rem} \label{rem 2.4.0}
i) We emphasize that the Markov property was essentially used to conlclude that the non-zero sub-invariant measure $f \cdot m$ constructed in the proof of Theorem \ref{thm 2.3}, v) $\Rightarrow$ i), is in fact invariant. However, it can be easily checked that if $(P_t)_t$ is sub-Markovian, then the condition $c((P_t)_t, m, (t_n)_n) < \liminf \frac{1}{t_n}\int_0^{t_n} m(P_s1) ds$ is sufficient for the existence of a non-zero sub-invariant finite measure $\rho \cdot m$.

ii) We would like to point out that although inequality (2.3) looks like a contraction assumption once we normalize the measure $m$ such that $m(E)=1$, a Banach fixed point type argument is rather inapplicable since $\mathcal{B} \ni A \mapsto \mathop{\sup}\limits_{n} \dfrac{1}{t_n} \int_{0}^{t_n} m(P_s1_A) \; ds$ is not a measure. 
\end{rem}

\begin{coro} \label{coro 2.4}

The following assertions are equivalent for a measurable Markovian transition function $(P_t)_{t \geq 0}$.

i) There exists a nonzero finite invariant measure.

ii) There exists a nonzero almost invariant measure.

\end{coro}

\begin{proof} 
The implication i) $\Rightarrow$ ii) is immediate and the converse follows by Theorem \ref{thm 2.3}.
\end{proof}

The next result reveals that under $(A_1)$ or $(A_2)$, if $m$ satisfies a sub-invariance property w.r.t. $(P_t)_{t \geq 0}$ only on a subset of $E$ of strictly positive measure then the existence of an invariant measure is ensured by Theorem \ref{thm 2.3}.

\begin{coro} Assume that $(P_t)_{t \geq 0}$ and $m$ are as in $(A_1)$ or $(A_2)$, and there exists $A \in \mathcal{B}$, $m(A) > 0$, such that
$$
m(1_A P_t 1_B) \leq m(B) \quad \mbox{for all} \; t\geq 0 \; {\rm and} \; B \in \mathcal{B}.
$$
Then $m$ is almost invariant.

\end{coro}

\begin{proof}
Let $A \in \mathcal{B}$ and $\varepsilon > 0$ such that $m(A) \geq \varepsilon$.
Then $m(P_t 1_B) \leq m(B) + m(1_{E \setminus A}P_t 1_B) \leq m(B) + m(1_{E \setminus A}) \leq m(B) + m(E)-\varepsilon$.
\end{proof}

\vspace{0.2 cm}

\noindent
{\bf Some versions of Theorem \ref{thm 2.3}}. In the end of this subsection we would like to formulate two versions of Theorem \ref{thm 2.3}, one in terms of resolvents and the other one involving a single operator $P$.
Their proofs are essentially the same as the one given for the main result, the only differences being that the semigroup property and the integrals are replaced either by the resolvent identity or by the Cesaro means of the powers $(P^n)_n$, and for this reason we omit them. 

First, for $(\alpha_n)_n \searrow 0$, define
$$
c((R_\alpha)_{\alpha > 0}, m, (\alpha_n)_n) := \mathop{\lim}\limits_{\varepsilon \searrow 0} \mathop{\sup}\limits_{\mathop{A \in \mathcal{B}}\limits_{m(A) < \varepsilon}} \mathop{\sup}\limits_{n} m(\alpha_n R_{\alpha_n}1_A).
$$

Also, $m$ is said {\it resolvent} almost invariant if $m(\alpha R_\alpha 1_A) \leq \phi(A) + \delta m(E)$ for all $A \in \mathcal{B}$ and $\alpha > 0$, where $\phi: \mathcal{B} \rightarrow [0, \infty)$ is absolutely continuous w.r.t $m$ in the sense made precise in the beginning of Subsection 2.2, and $\delta \in [0, 1)$. 

With Remark \ref{prop 2.3}, i) in mind, we have:

\begin{thm} \label{thm 2.5}
Let $m$ be a finite positive measure on $(E, \mathcal{B})$ such that $(P_t)_{t \geq 0}$ satisfies (A$_2$) w.r.t. $m$. 
The following assertions are equivalent.

i) There exists a non-zero finite $(P_t)_{t \geq 0}$-invariant measure which is absolutely continuous w.r.t. $m$.

\vspace{0.2cm}

ii) The measure $m$ is resolvent almost invariant.
\vspace{0.2 cm}

iii) There exists $(\alpha_n) \searrow 0$ such that 
$$
c((R_\alpha)_{\alpha}, m, (\alpha_n)_n) < m(E).
$$
\end{thm}
 
We turn now to the case of a single operator.
Analogously to conditions A$_1$ and A$_2$, for an operator $P$ we shall assume that it is either a Markovian operator on $L^p(m)$, $1 \leq p < \infty$, or a Markovian kernel which respects the $m$-classes, that is the measure $m \circ P$ is absolutely continuous with respect to $m$. 
Also, we say that $m$ is almost invariant, resp. mean almost invariant for the operator $P$ if $m(P^n(A)) \leq \phi(A) + \delta m(E)$, resp. $m(S_n1_A) \leq \phi(A) + \delta m(E)$ for all $n$ greater than some $n_0$, where $\phi$ and $\delta$ are as for relation (2.1), and the operators $S_n$ are given by
$$
S_n = \dfrac{1}{n}\sum\limits_{k=0}^{n-1} P^k.
$$ 
The index $c$ is defined by $c(P, m):= \mathop{\lim}\limits_{\varepsilon \searrow 0} \mathop{\sup}\limits_{\mathop{A \in \mathcal{B}}\limits_{m(A) \leq \varepsilon}} \mathop{\sup}\limits_{n \geq 1} m(S_n 1_A)$.

Now, Theorem \ref{thm 2.3} stated for a single operator $P$ reads:

\begin{thm} \label{thm 2.3.1} %theorem 2.4

The following assertions are equivalent.

i) There exists a non-zero finite invariant measure for $P$ which is absolutely continuous with respect to $m$

\vspace{0.2cm}

ii) The measure $m$ is almost invariant for $P$.

\vspace{0.2cm}

iii) The measure $m$ is mean almost invariant for $P$.

\vspace{0.2cm}

iv) $c(P, m) < m(E)$. 

\end{thm} 

\subsection{Construction of auxiliary measures}
Throughout, $(P_t)_{t\geq 0}$ is a measurable Markovian transition function on $(E, \mathcal{B})$.
Going back to assumption $(A_2)$ it is clear that when we want to apply Theorem \ref{thm 2.3} to  $(P_t)_{t\geq 0}$, the first step is to look for an auxiliary measure on $(E, \mathcal{B})$, i.e. a measure with respect to which $(P_t)_{t\geq 0}$ respects classes.
As in \cite{BeBo04} or \cite{RoTr07}, it turns out that the resolvent provides a natural way to construct such measures, as follows.

\begin{prop} \label{prop 2.7}
Let $\mu$ be a probability measure on $(E, \mathcal{B})$ and for any fixed $\alpha > 0$ define the finite positive measure $m$ by
$$
m(f) = \mu \circ R_{\alpha} f = \int_E R_{\alpha} f d\mu \leqno(2.4)
$$
for all positive and $\mathcal{B}$-measurable functions $f$. Then $m$ is an auxiliary measure for $(P_t)_{t \geq 0}$.
\end{prop}

\begin{proof} 
If $A \in \mathcal{B}$ such that $m(A)=0$ then $m(P_t 1_A) = \int_E R_{\alpha} (P_t1_A) d\mu= e^{\alpha t}(\int_E (R_\alpha 1_A-\int_0^te^{-\alpha s} P_s 1_A ds) d \mu ) \leq e^{\alpha t}\int_E R_\alpha 1_A d \mu = e^{\alpha t} m(A) = 0$.
\end{proof}

\begin{rem}  \label{rem 2.8}
i) If $E$ is a separable metric space and $\mathop{\lim}\limits_{t \to 0} P_tg(y) = g(y)$ for any bounded continuous function $g$ on $E$, $y \in E$, we get additional information about some particular measures constructed by Proposition \ref{prop 2.7}, namely topological full support. 
More precisely, if $(x_n)_{n \geq 1}$ is a dense subset in $E$, then the measure $m = \mu \circ R_{\alpha}$, where $\mu = \mathop{\sum}\limits_{n = 1}^{\infty} \dfrac{1}{2^n} \delta_{x_n}$, has full support for all $\alpha > 0$. 
Moreover, one can associate a generalized Dirichlet form on $L^2(m)$ such that the associated semigroup is an $m$ - version of $(e^{-\alpha t}P_t)_{t > 0}$. For these results we refer to \cite{RoTr07}, Lemma 2.3 and Theorem 3.2.
As we shall see later, the problem of existence of invariant measures can be approached in terms of sectorial forms via functional inequalities.

ii) When we deal with a single Markovian kernel $P$ and $\mu$ is a probability measure on $(E, \mathcal{B})$, then one has a similar construction for an auxiliary measure for $P$, by setting $m:=\mu \circ R$, where $R$ is the {\it resolvent} kernel $R:=\mathop{\sum}\limits_{n=0}^\infty \dfrac{1}{2^{n+1}}P^n$.

\end{rem}

If we consider transition functions satisfying (A$_2$) with respect to a measure given by (2.4), then we have the following sufficient condition for (2.3).

\begin{thm} \label{thm 2.9}
Let $\mu$ be a probability measure on $(E, \mathcal{B})$.

i) If there exists $(t_n)_n \nearrow \infty$ such that
$$
\widetilde{c} := \mathop{\lim}\limits_{\varepsilon \searrow 0} \mathop{\sup}\limits_{\mathop{A \in \mathcal{B}}\limits_{\mu \circ R_{\alpha}(A) < \varepsilon}} \mathop{\sup}\limits_{n} \dfrac{1}{t_n} \int_{0}^{t_n} \mu (P_s1_A) ds < \alpha,
$$
then there exists $(t'_n)_n$ s.t. $c((P_t)_t, \mu \circ R_{\alpha}, (t'_n)_n) < \mu \circ R_{\alpha}(E) = \dfrac{1}{\alpha}$.

ii) If there exists $(\alpha_n)_n \searrow 0$ such that
$$
\mathop{\lim}\limits_{\varepsilon \searrow 0} \mathop{\sup}\limits_{\mathop{A \in \mathcal{B}}\limits_{\mu \circ R_{\alpha}(A) < \varepsilon}} \mathop{\sup}\limits_{n} \mu \circ \alpha_n R_{\alpha_n}1_A < \alpha,
$$
then there exists $(\alpha'_n)_n$ s.t. $c((R_{\beta})_{\beta}, \mu \circ R_{\alpha}, (\alpha'_n)_n) < \dfrac{1}{\alpha}$.
\end{thm}

\begin{proof} We treat only the first case. 
Let $m = \mu \circ R_{\alpha}$ and shift $(t_n)_n$ into some $(t'_n)_n$ such that $t'_0 > \dfrac{1}{\alpha^2(1 - \widetilde{c})}$. 
Then
$$
c((P_t)_t, m, (t'_n)_n) = \mathop{\lim}\limits_{\varepsilon \searrow 0} \mathop{\sup}\limits_{\mathop{A \in \mathcal{B}}\limits_{m(A) < \varepsilon}} \mathop{\sup}\limits_{n} \dfrac{1}{t'_n} \int_{0}^{t'_n} m (P_s1_A) ds 
$$
$$
= \mathop{\lim}\limits_{\varepsilon \searrow 0} \mathop{\sup}\limits_{\mathop{A \in \mathcal{B}}\limits_{m(A) < \varepsilon}} \mathop{\sup}\limits_{n} \mu (\int_{0}^{\infty} e^{-\alpha t} \dfrac{1}{t'_n} \int_{0}^{t'_n} P_{r + s}1_A ds dr) 
$$
$$
= \mathop{\lim}\limits_{\varepsilon \searrow 0} \mathop{\sup}\limits_{\mathop{A \in \mathcal{B}}\limits_{m(A) < \varepsilon}} \mathop{\sup}\limits_{n} \mu(\dfrac{1}{\alpha} \dfrac{1}{t'_n} \int_{0}^{t'_n} P_s1_A ds \; + \dfrac{1}{t'_n} \int_{0}^{t'_n} e^{-\alpha t} \int_{t'_n}^{t'_n + t} P_s1_A ds 
$$
$$
- \dfrac{1}{t'_n} \int_{0}^{t'_n} e^{-\alpha r} \int_{0}^{r} P_s1_A ds dr).
$$

Integrating by parts the last two terms in the last bracket, we obtain
$$
c((P_t)_t, m, (t'_n)_n) = \mathop{\lim}\limits_{\varepsilon \searrow 0} \mathop{\sup}\limits_{\mathop{A \in \mathcal{B}}\limits_{m(A) < \varepsilon}} \mathop{\sup}\limits_{n} (\dfrac{1}{t'_n} \int_{0}^{t'_n} \mu (P_s1_A) ds 
$$
$$
+ \dfrac{1}{t'_n} \mu (P_{t'_n} R_{\alpha} 1_A) \; - \dfrac{1}{t'_n} m(A)) \leq \dfrac{\widetilde{c}}{\alpha} + \dfrac{1}{\alpha^2t'_0} < \dfrac{1}{\alpha.}
$$
\end{proof}

\section{Applications}

In the sequel we will apply the main results of the previous two sections in several directions. 

First, we present a comparison of the measure theoretic conditions assumed in Theorem \ref{thm 2.3} with the purely topological ones which appear in the classical results of Krylov-Bogolibov or Lasota and Szarek concerning the existence of invariant measures for Feller transition functions on Polish spaces. 
Although not in its full generality, we give a very short proof of the latter above mentioned result in terms of almost invariant measures. As a benefit of this approach, we also have regularity for the obtained invariant measure.

Secondly, we take another look at some versions of Harris' ergodic theorem and give short proofs for the existence of invariant measures under more general conditions.
We also investigate the number of the othogonal invariant (resp. ergodic) measures. 
This approach allows us to give an answer to the open question mentioned by Tweedie (see \cite{Tw01}, Remark 6.) concerning the sufficiency of the so called {\it generalized drift condition} for the existence of an invariant measure.

In a third part we deal with nonlinear SPDEs on a Gelfand triple $V \subset H \subset V^\ast$. 
We show that under a Wang's Harnack type inequality, the strict coercivity condition with respect to the $H$-norm is sufficient to guarantee the existence of a unique invariant probability measure for the solution.
In order to justify even more our two-step approach, we apply it to a perturbation of a Markov kernel satisfying a combined Harnack-Lyapunov condition, for which the result of Tweedie (Theorem \ref{thm 3.5} below) can not be used. 
We also discuss the applicability of Harris' result to this kind of perturbation.

At the end of this section we present several applications to sectorial forms, mainly in terms of functional inequalities. In this situation we remain in the case when the constant $\delta$ (and hence the index $c$) in (2.1) equals $0$, hence we do not exploit the fact that Theorem \ref{thm 2.3} allows us to drop the uniform integrability down to $c((P_t)_t, m) < m(E)$.

\subsection{ Almost invariant measures and the theorem of Lasota and Szarek}
Along this subsection we assume that $(P_t)_{t \geq 0}$ is a measurable Markovian transition function on a Polish space $E$ with its corresponding Borel $\sigma$-algebra $\mathcal{B}$.
Recall that the classical Krylov-Bogoliubov theorem asserts that if $(P_t)_{t \geq 0}$ has the Feller property (i.e. $P_t (C_b) \subset C_b$, where $C_b$ denotes the space of bounded and continuous real functions on $E$) and there exist a probability $\mu$ and a sequence $(t_n)_n \nearrow \infty$ such that the family of probability measures $(\mu_n)_{n \geq 1}$ defined by
$$
\mu_n(A) := \dfrac{1}{t_n}\int_{0}^{t_n} \mu (P_s1_A) ds, \; A \in \mathcal{B}, \; n \geq 1, \leqno(3.1)
$$
is tight, then there exists an invariant probability measure $\nu$ which is the limit of some weakly convergent subsequence $(\mu_{n_k})_{k \geq 1}$.

At least when $(P_t)_{t \geq 0}$ (or $(R_{\alpha})_{\alpha > 0}$) has the strong Feller property, i.e. $P_t$ (resp. $R_\alpha$) maps bounded Borel functions on $E$ into continuous functions, the above mentioned result can be compared with Theorem \ref{thm 2.3} as follows.

\begin{prop} \label{prop 3.1}
Assume that $(P_t)_{t > 0}$ is strong Feller and there exist a probability measure $\mu$ and a sequence $(t_n)_n \nearrow \infty$ such that $(\mu_n)_n$ defined by (3.1) is tight. 
Then for all $\alpha > 0$, if $m = \mu \circ R_{\alpha}$ then there exists a subsequence $(t_{n_k})_{k \geq 1}$ such that
$$
c((P_t)_{t > 0}, m, (t_{n_k})_{k}) = 0.
$$
\end{prop}

\begin{proof} Let $(t_{n_k})_{k \geq 1}$ such that $(\mu_{n_k})_{k \geq 1}$ converges weakly to an invariant measure $\nu$. 
Notice first that $\nu \ll m$. 
Indeed, if $m(A) = 0$, then $\nu(A) = \alpha \nu(R_{\alpha}1_A)  = \alpha \mathop{\lim}\limits_{k} \dfrac{1}{t_{n_k}}\int_{0}^{t_{n_k}} m(P_s1_A) ds = 0$, hence $\nu = \rho \cdot m$ for some $\rho \in L^1(m)$, where the second equality follows by the strong Feller property and the weakly convergence of $(\mu_{n_k})_k$.

On the other hand,
$ m(\rho g) = \nu(g) = \alpha \nu(U_{\alpha}g) = \alpha \mathop{\lim}\limits_{n} \dfrac{1}{t_{n_k}} \int_{0}^{t_{n_k}} m(P_sg) ds = \alpha \mathop{\lim}\limits_{k} m(g \cdot \dfrac{1}{t_{n_k}} \int_{0}^{t_{n_k}} P_s^{\ast}1 ds)$ for all $g \in L^{\infty}(m)$, when $(P_s^{\ast})_s$ is the adjoint of $(P_s)_s$ with respect to $m$.
It follows that $\left(\dfrac{1}{t_{n_k}}\int_{0}^{t_{n_k}} P_s^{\ast}1 ds\right)_{k \geq 1}$ is weakly convergent to $\dfrac{\rho}{\alpha}$ in $L^1(m)$, hence $c((P_t)_{t > 0}, m, (t_{n_k})_k) = 0$ by Dunford-Pettis theorem.
\end{proof}

We recall that by \cite{LaSz06}, Proposition 3.1, if $E$ is a Polish space, $\mathcal{B}$ is its Borel $\sigma$-algebra, and $(P_t)_{t \geq 0}$ is a measurable Markovian transition function possessing the Feller property such that there exist a compact set $K \subset E$ and $x \in E$ for which
$$
\mathop{\limsup\limits_{t \to \infty}}\dfrac{1}{t}\int_0^t\delta_x \circ P_s1_Kds = \epsilon > 0, \leqno{(3.2)}
$$
then there exists a nonzero invariant measure for $(P_t)_{t \geq 0}$;
this existence result dates back to the work of \cite{Fo69} and \cite{St86};
see also \cite{KoPeSz10} for uniqueness of the invariant measure, when $(P_t)_{t\geq 0}$ satisfies an additional equi-continuity property. 

\begin{prop} \label{prop 3.2}
Let $(P_t)_{t \geq 0}$ be a measurable Markovian transition function on a Polish space $E$ with $\mathcal{B}$ its Borel $\sigma$-algebra.
If condition (2.3) holds then (3.2) is satisfied with $\delta_x$ replaced by $m$.
\end{prop}

\begin{proof}
Let $m$ be a probability measure on $(E, \mathcal{B})$ such that (2.3) holds, i.e. there exists $\varepsilon > 0$ such that
$\mathop{\sup}\limits_{\mathop{A \in \mathcal{B}}\limits_{m(A) \leq \varepsilon}} \mathop{\sup}\limits_{n} \dfrac{1}{t_n} \int_{0}^{t_n} m(P_s1_A) \; ds < 1$.
Since $m$ is tight, there exists a compact set $K \subset E$ such that $m(E \setminus K) < \varepsilon$. 
Therefore, $\mathop{\sup}\limits_{n} \dfrac{1}{t_n} \int_{0}^{t_n} m(P_s1_{E \setminus K}) \; ds < 1$ and because $P_t1=1,\; t\geq 0$ it follows that
$\mathop{\inf}\limits_{n} \dfrac{1}{t_n} \int_{0}^{t_n} m(P_s1_{K}) \; ds > 0$.
\end{proof}

Under the asssumptions of Proposition \ref{prop 3.2} and if $(P_t)_t$ is Feller, the existence of a nonzero invariant measure follows by \cite{LaSz06}, since it is straightforward to check that the result of \cite{LaSz06} still holds true, with a similar proof as the original, once we replace the Dirac measure $\delta_x$ with an arbitrary probability measure (in our case $m$) on $(E,\mathcal{B})$.

We emphasize that, in contrast with our result, besides the additional regularity conditions for $(E, \mathcal{B})$ and $(P_t)_{t \geq 0}$ which have to be assumed, by the result of Lasota and Szarek it does not follow that the obtained invariant measure is absolutely continuous w.r.t. $m$.

Although Theorem \ref{thm 2.3} deals with non-topological assumptions, we will see in the sequel that it works pretty well in combination with Prohorov's theorem in order to provide a quite short and transparent proof for Lasota's result on the existence of invariant measures when $\epsilon$ in (3.2) is strictly grater than $\dfrac{1}{2}$. 
In addition, our approach guarantees the absolute continuity of the obtained invariant measure with respect to a convenient auxiliary measure.
To do this, first note that if $(P_t)_t$ is a measurable transition function on a Polish space $(E, \mathcal{B})$, $\nu$ is a probability on $(E, \mathcal{B})$, and $K$ is a compact subset of $E$, then for $(t_n)_n \nearrow \infty$, the family of measures $(\mu_n^K)_n$ defined by
$$
\mu_n^K(f):=\dfrac{1}{t_n}\int_0^{t_n} \nu(P_s(1_K f)) ds
$$
is tight. 
Let $(t_{n_k})_k \nearrow \infty$ and $\mu$ be such that $(\mu^K_{n_k})_k$ is weakly convergent to $\mu$. 
Even if condition (3.2) (where $\delta_x$ is replaced by $\nu$) ensures the non-triviality of $\mu$, the difficulty is that it will no longer be invariant. 
Lasota and Szarek avoided this impediment by looking at a Riesz type decomposition for the positive functional on $C_b$ obtained as a Banach limit of the genuine Krylov-Bogoliubov measures given by (3.1). 
However, as it is shown by the next result, it turns out that the auxiliary measure defined for some arbitrarily fixed $\alpha > 0$ by
$$
m(f)= \mu \circ \alpha R_\alpha \; \mbox{ for all} \; f \in p\mathcal{B} \leqno{(3.3)}
$$   
is almost invariant.

\begin{thm} \label{thm 3.2} %theorem 3.2
If $(P_t)_{t \geq 0}$ is a measurable Markovian transition function on $(E, \mathcal{B})$ 
possessing the Feller property such that there exist a compact set $K \subset E$ and a probability measure $\nu$ for which it holds that
$$
\mathop{\limsup\limits_{t \to \infty}}\dfrac{1}{t}\int_0^t \nu(P_s1_K)ds > \dfrac{1}{2} \;,
$$
then $m$ given by (3.3) is almost invariant for $(P_t)_{t \geq 0}$. 
Hence, there exists an invariant measure which is absolutely continuous with respect to $m$.
\end{thm}

\begin{proof}
By hypothesis, there exists $(t_n)_n \nearrow \infty$ such that $\mu_n^K(E)=\dfrac{1}{t_n}\mathop{\int}_0^{t_n} \nu(P_s1_K)ds > \dfrac{1}{2}$, therefore we also have that $m(E)=\mu(E) > \dfrac{1}{2}$.
If $0 \leq f \in C_b$, then 
$$
\mu(P_t f) = \lim\limits_k \mu_{n_k}^K(P_t f)
=\lim\limits_k\dfrac{1}{t_{n_k}}\int_0^{t_{n_k}} \nu(P_s(1_K P_tf))ds 
$$
$$
\leq \limsup\limits_k\dfrac{1}{t_{n_k}}\int_0^{t_{n_k}} \nu(P_{s+t} f))ds
\leq \limsup\limits_k\dfrac{1}{t_{n_k}}\int_t^{t_{n_k}+t} \nu(P_{s} f))ds
$$
$$
= \limsup\limits_k[\dfrac{1}{t_{n_k}}(\int_0^{t_{n_k}} \nu(P_{s} f)ds + \int_{t_{n_k}}^{t_{n_k}+t} \nu(P_{s} f)ds - \int_{0}^{t} \nu(P_{s} f))ds)]
$$
$$
\leq \limsup\limits_k\dfrac{1}{t_{n_k}}\int_0^{t_{n_k}} \nu(P_{s} f)ds
$$
$$
\leq \limsup\limits_k[\dfrac{1}{t_{n_k}}(\int_0^{t_{n_k}} \nu(P_{s}(1_K f))ds+\mathop{\int}_0^{t_{n_k}} \nu(P_s1_{E \setminus K})ds)]
%$$
%$$
\leq \mu(f) + \dfrac{1}{2} = \mu(f) + \dfrac{1}{2\mu(E)}\mu(E).
$$
Replacing $f$ by $\alpha R_\alpha f$ and using an approximation argument, it follows that $m(P_t 1_A) \leq m(A) + \dfrac{1}{2m(E)}m(E)$ for all $A \in \mathcal{B}$. But as we noticed at the beginning, $m(E)>\dfrac{1}{2}$, hence $m$ is almost invariant.
\end{proof}

\subsection{Almost invariant measures and Harris' theorem}

In this subsection, in contrast with the previous one, we investigate some results concerning invariant measures which involve exclusively non-topological conditions. 
Therefore, we place ourselves in the general situation of a Markovian kernel $P$ on a measurable space $(E, \mathcal{B})$.
We emphasize that, in view of Proposition \ref{prop 2.3}, all of the following results, although stated for a single operator, can be applied to the case of a continuous time transition function $(P_t)_{t\geq 0}$ just by looking at a single kernel $P_{t_0}$; see Subsection 3.3. 

We first recall several definitions and conditions required by some well known versions of Harris' theorem to guarantee the existence, uniqueness, and also different rates of stability (polynomial, sub-exponential or exponential) for a semigroup. 
These conditions slightly differ one from another but, in principle, they assume the existence of a {\it small} set (in the sense made precise below) which is visited infinitely often. 
Small sets should be regarded as a substitute for infinitely often visited compact sets in the Feller case (which is the situation of the theorem of Lasota and Szarek discussed in the previous subsection). 
As a matter of fact, if $P$ is irreducible and a $T$-chain, then every compact set is a small set; see \cite{MeTw93a}.
In practice, the small sets of interest are the sub-level sets of a {\it Lyapunov} function.

Recall that (cf. e.g. \cite{MeTw93a}, Chapter 5, Section 5.2) a measurable set $C \in \mathcal{B}$ is {\it small} with respect to a Markovian kernel $P$ on $(E, \mathcal{B})$ if there exist a constant $\alpha \in (0,1]$ and a probability measure $\nu$ such that 
$$
\inf\limits_{x \in C}P(x, \cdot) \geq \alpha \nu(\cdot).
$$

Let us recall the following two assumptions; see e.g. \cite{HaMa11}.

\vspace{0.2cm}
\noindent
{\bf Assumption {\bf A}.} 
There exist a function $V \in p\mathcal{B}$ and constants $b \geq 0$ and $\gamma \in (0, 1)$ such that
$$
PV \leq \gamma V + b \;\;\; {\rm on} \; E.
$$  
Such $V$ is usually called a Lyapunov or Foster-Lyapunov function.
\noindent
Furthermore, the sub-level set $[V \leq r]$ is small for some $r > 2b \slash (1-\gamma)$.

\vspace{0.2cm}

\noindent
{\bf Assumption {\bf A'}.} 
There exist $\widetilde{V} \in p\mathcal{B}$, $\widetilde{V} \geq 1$, constants $\widetilde{b} \geq 0$ and $\widetilde{\gamma} \in (0, 1)$, and a subset $S \subset E$ which is small such that
$$
P\widetilde{V} \leq \widetilde{\gamma} \widetilde{V} + \widetilde{b} 1_S \;\;\; {\rm on} \; E.
$$  

The second assumption is encountered more frequently in the theory of Markov chains and in general it does not imply the first one; see, e.g. \cite{HaMa11}, Remark 3.3. 
However, it was shown in \cite{HaMa11}, Theorem 3.4, that if Assumption {\bf A'} holds for $P$, then Assumption {\bf A} holds for $S_N = \dfrac{1}{N}\sum\limits_{k=0}^{N-1} P^k$ for some sufficiently large $N$.

It is well known that under Assumption {\bf A} not only existence and uniqueness of the invariant measure is ensured, but also the spectral gap in a weighted supremum norm.
For completeness we state this result below.
Although there exist several different approaches to prove it, we refer the reader to the work of \cite{HaMa11} for a direct proof based on Banach fixed point theorem; see also the references therein.

\begin{thm} \label{thm 3.4}
{\it(cf. e.g. \cite{HaMa11})} If Assumption {\bf A} is satisfied, then there exists a unique invariant probability measure $m$ for $P$.
 In addition, for some constants $C>0$ and $\gamma \in (0,1)$ it holds that
$$
\|P^nf - m(f)\|\leq C \gamma^n  \|f-m(f)\|
$$ 
for all $\mathcal{B}$-measurable $f$ with $\|f\|<\infty$, where $\|f\|=\|\dfrac{f}{1+V}\|_\infty$.
\end{thm}

There is an extended notion of small sets, namely the so called {\it petite sets}, which are defined by means of {\it generalized} resolvents.
These instruments were developed by Meyn and Tweedie in order to study (geometric) convergence for Markov processes in both discrete and continuous time, and we refer the reader to \cite{MeTw93a}, \cite{MeTw93b}, \cite{MeTw93c}, and \cite{MeTw93d}.
For a study of weaker rates of convergence we mention \cite{DoFoGu09}, \cite{Ha10}, and the references therein.

Anyway, to check the smallness of a set $C$ is a quite delicate issue and the usual techniques require continuity or irreducibility conditions for the associated Markov process.
In the papers \cite{Tw01} and \cite{FoTw01}, the authors investigate the existence of invariant measures for Markov chains, with direct applications to non-linear time series, assuming the existence of a Foster-Lyapunov function and, instead of the smallness property for the test set $C$, a weak uniform countable additivity condition.
More precisely, the following assumption has been considered.

\vspace{0.2cm}
\noindent
{\bf Assumption {\bf B}.} i) There exist a measurable function $V:E \rightarrow [0, \infty)$, a finite constant $b$ and a measurable set $C$ such that
$$
PV \leq V -1 + b1_C \;\;\; {\rm on} \; E.
$$

ii) The set $C$ from i) is such that the following uniform countable additivity condition holds: for all $(A_n)_n \subset \mathcal{B}$ decreasing to $\emptyset$ we have that
$$
\lim\limits_n\sup\limits_{x\in C}P(A_n)(x)=0.
$$
Under such hypotheses, Tweedie proved the following result.

\begin{thm} \label{thm 3.5}
{\it (cf. \cite{Tw01}, Theorem 1)} If Assumption {\bf B} holds, then there exists a positive finite number of orthogonal invariant probability measures $\nu_i, 1\leq i \leq n$.
Moreover, for each $x\in E$ there exists a convex combination $m$ of $(\nu_i)_i$ such that
$$
\dfrac{1}{n}\sum\limits_{k=1}^nP^k(x,A) \mathop{\longrightarrow}\limits_n m(A),
$$ 
for all $A \in \mathcal{B}$.

\end{thm}

\begin{rem}
i) The uniform countable additivity condition looks easier to check than the smallness property, since e.g. it is clearly satisfied if there exists a finite measure $\nu$ such that $P(x, \cdot) \leq \nu(\cdot)$ for all $x \in C$; see \cite{Tw01}, Remark 5 for more details.

ii) We stress out that in all of the above assumptions one can let $V$ take infinite values because we may consider the restriction of $P$ to the absorbing set $[V<\infty]$ without altering the other conditions.  
\end{rem}

\noindent
{\bf Open question} {\it (cf. \cite{Tw01}, Remark 6):} 
Can we replace the constant $b$ in Assumption {\bf B}, i) by a not necessarily bounded function?

\vspace{0.2cm}
For the rest of this subsection, our main purpose is to recapture the above discussed versions of Harris's result in a single more general statement with a very short proof in terms of almost invariant measures, and also to give an answer for the open question.

For convenience, we denote by $\mathcal{B}_1^+$ the set of all 
$\mathcal{B}$-measurable real-valued functions $f$ such that $0 \leq f \leq 1$, and recall that in Section 2 we introduced the operators $S_n$ and $R$ by setting
$$
S_n = \dfrac{1}{n}\sum\limits_{k=0}^{n-1} P^k, \; \mbox{resp.} \; R = \sum\limits_{k=0}^\infty\dfrac{1}{2^{k+1}}P^k.
$$ 
Let us introduce the following assumptions.

\vspace{0.2cm}
\noindent
{\bf Assumption {\bf C}.} i) There exist $C \in \mathcal{B}$, $\phi:\mathcal{B}_1^+ \rightarrow \mathbb{R}_+$, and $\gamma : E \rightarrow \mathbb{R}_+$ such that
$$
Pf(x) \leq \phi (f) + \gamma (x)
$$ 
for all $x \in C$ and $f \in \mathcal{B}_1^+$.

ii) There exists a finite positive measure $m$ on $E$ such that 

ii.1) $\phi \circ R \ll m \circ R$ (i.e. $\mathop{\lim}\limits_{m \circ R(A) \to 0} \phi \circ R (1_A) = 0$).

ii.2) There exists $n_0 > 0$ such that $\mathop{\sup}\limits_{n \geq n_0} m( S_n (1_C(\gamma-1))) < 0$, where $C$ is the set from i). 

It is convenient to look for $\phi$ which is a real function composed with a measure. 
Also, if $\gamma$ is constant then additional information about the number of the orthogonal measures can be obtained.
For these reasons, we shall consider the following particular case of Assumption {\bf C}.

\vspace{0.2cm}
\noindent
{\bf Assumption C'.} i) There exist a finite positive measure $m$ on $E$, a set $C\in \mathcal{B}$, a function $\phi:[0, \infty) \rightarrow [0, \infty)$ which is continuous and null in $0$, and $\delta \in [0,1)$ such that
$$
Pf(x) \leq \phi(m(f)) + \delta
$$
for all $f \in \mathcal{B}_1^+$ and $x \in C$.

ii) There exists $n_0 > 0$ such that $\mathop{\inf}\limits_{n \geq n_0} m(S_n 1_C)>0$, where $C$ is the set from i).

The following result states that Assumptions {\bf A}, {\bf A'}, and {\bf B} are particular cases of Assumption {\bf C'}.

\begin{prop} \label{prop 1} The following assertions hold.

i) If Assumption {\bf A} is satisfied, then for all $N>0$ there exists $n_0 > 0$ and $\delta \in [0,1)$ such that
$$
P^nf(y) \leq P^mf(x) + \delta
$$
for all $n,m \geq n_0$, $x,y \in [V<N]$, and $f\in \mathcal{B}_1^+$.

\vspace{0.2cm}
ii) If Assumption {\bf A'} is satisfied, then for all $N>0$ there exist $n_0 > 0$ and $\delta \in [0,1)$ such that
$$
S_nf(y) \leq S_mf(x) + \delta
$$
for all $n,m \geq n_0$, $x,y \in [V<N]$, and $f\in \mathcal{B}_1^+$.

\vspace{0.2cm}
In particular, if any of the Assumptions {\bf A} or {\bf A'} is satisfied, then Assumption {\bf C'} holds for all $P^n$ resp. $S_n$ if $n$ is sufficiently large. 

\vspace{0.2cm}
iii) If Assumption {\bf B}, i) is satisfied, then Assumption {\bf C'}, ii) holds for every (non-trivial) measure $m$.
\end{prop}

\begin{proof} We shall prove only i) and iii), since the second assertion can be easily proved using the same ideas involved for proving the other two.

i). Iterating the relation $PV \leq \gamma V+b$, we get that for $n > 0$, $P^nV \leq \gamma^nV + \dfrac{b}{1-\gamma}$ and 
$$
P^n([V>r]) \leq \dfrac{1}{r}P^nV \leq \dfrac{1}{r}(\gamma^nV + \dfrac{b}{1-\gamma}) \leq \dfrac{\gamma^nV}{r} + \dfrac{1}{2}. \leqno{(\ast)}
$$
On the other hand, we know that $C:=[V \leq r]$ is small, so there exist a constant $\alpha \in (0, 1]$ and a probability $\nu$ such that
$Pf(y)\geq \alpha \nu(f)$ for all $y \in C$ and $f \in \mathcal{B}_1^+$.
Taking in this inequality $1-f$ instead of $f$, we obtain $Pf(y) \leq 1-\alpha +\alpha\nu(f)$ for all $y \in C$.
Combining the last two inequalities it follows that $Pf(x)\leq Pf(y)+1-\alpha$
for all $x,y \in C$, hence
$$
Pf \leq Pf(y)+1-\alpha 1_C \;\;\; {\rm on} \; E
$$
for all $y \in C$.
Integrating this inequality w.r.t $P^{n-1}(x,\cdot)$, $x\in E$ we obtain
$$
P^nf \leq Pf(y) + 1 - \alpha P^{n-1}1_C \;\;\; {\rm on} \; E 
$$
for all $y \in C$ and $n>0$.
Replacing $f$ with $1-f$ we get
$$
Pf \leq P^nf(x) + 1 - \alpha P^{n-1}1_C(x)1_C 
$$
for all $x \in E$,
and again integrating the last inequality but now w.r.t $P^{m-1}(y,\cdot)$ we obtain
$$
P^mf(y) \leq P^nf(x) + 1 - \alpha P^{n-1}1_C(x)P^{m-1}1_C(y)
$$
for all $x,y\in E$, $f\in \mathcal{B}_1^+$, and $n, m >0$.
Now, the assertion follows if we combine the last inequality with relation ($\ast$), since the coefficient 
of $\alpha$ is far away from $0$ for all $n$ and $m$ sufficiently large, uniformly in $x,y \in [V < N]$.

The fact that Assumption {\bf A} implies {\bf C'} for $P^n$ follows by choosing $\phi(x)=x$, $C=[V\leq r]$, and $m=\delta_y \circ P^n$ for some arbitrarily fixed $y \in C$, and taking into accout relation ($\ast$).

\vspace{0.2cm}
iii). Let $\mu$ be a non-zero finite measure.
Since $V < \infty$, there exists $n_0>0$ such that $\mu([V\leq n_0])\geq \varepsilon >0$.
Now, iterating the relation $PV \leq V-1+b1_C$ we get that $P^nV \leq V-n+ b\sum\limits_{k=0}^{n-1}P^k(1_C)$, hence 
$S_n(1_C)\geq \dfrac{1}{b}(1-\frac{V}{n})$ and therefore
$$
1_{[V\leq n_0]}S_n(1_C)\geq \dfrac{1}{2b}1_{[V\leq n_0]}
$$
for all $n\geq 2n_0$. 
Integrating the last inequality with respect to $\mu$ we conclude that
$$
\inf\limits_{n\geq 2n_0}\mu(S_n(1_C))\geq \dfrac{1}{2b}\mu([V\leq n_0]) \geq \dfrac{\varepsilon}{2b}>0
$$
which proves the assertion.

\end{proof}

\begin{rem} \label{rem 3.8}
Often, the sublevel sets $[V\leq r]$ of the Lyapunov function $V$ which appears in Assumption {\bf A} are small for all sufficiently large $r$.
In this case, one can easily adapt the proof of Proposition \ref{prop 1}, i) to show that Assumption {\bf C'} holds for $P$, not just for $P^n$ with $n$ big enough.

\end{rem}

We can now state the main result of this subsection.

\begin{thm} \label{thm 3.8.0} %Theorem 3.8
If Assumption {\bf C} is satisfied, then $m\circ R$ is mean almost invariant.
\end{thm}

\begin{proof}
With the set $C$ given by the hypothesis, we have for all $f \in \mathcal{B}^+_1$ that
$Pf \leq \phi(f) + \gamma 1_C + 1_{E \setminus C}$ which leads to 
$$
P^kf \leq \phi(f) + P^{k-1}(\gamma 1_C) + P^{k-1}1_{E\setminus C}, \;\;\; k>0.
$$
Considering the Cesaro means, we obtain that
$$
S_nf = \dfrac{1}{n}\sum\limits_{k=0}^{n-1}P^kf \leq \dfrac{1}{n}f + \dfrac{n-1}{n}\phi(f)+\dfrac{n-1}{n}S_{n-1}(\gamma 1_C) + \dfrac{n-1}{n}S_{n-1}(1_{E\setminus C})
$$
$$
\leq \dfrac{1}{n} + \phi(f) + S_{n-1}(1_C(\gamma-1)) + 1
$$
Integrating with respect to $m$ it leads to
$$
m(S_nf)\leq m(E)\phi(f) + \{m(S_{n-1}(1_C(\gamma-1)))m(E)^{-1} + (1+\frac{1}{n})\}m(E).
$$
Now by hypothesis, the term in brackets is strictly less then $1$ for all sufficiently large $n$, uniformly in $n$.
Hence there exist $\delta \in [0,1)$ and $n_0$ such that for all $n \geq n_0$ we have
$$
m(S_nf) \leq m(E)\phi(f) + \delta m(E).
$$
By replacing $f$ with $Rf$ in the last inequality we obtain for all $f\in \mathcal{B}_1^+$ and $n \geq n_0$ that
$$
m\circ R(S_nf) \leq m(E)\phi\circ R(f) + \delta m\circ R(E).
$$
Taking into account Remark \ref{rem 2.8}, ii), it follows that $m\circ R$ is mean almost invariant.

\end{proof}

We recall the following condition.

\vspace{0.2cm}
\noindent
{\bf Generalized drift condition.} There exist two measurable functions $V,b:E \rightarrow [0, \infty)$, and a measurable set $C$ such that
$$
PV \leq V  -1 + b 1_C \, \mbox{ on } C.
$$
%for all $x \in C$.

\noindent
Next, we consider an integrability assumption for $b$ that appears in the generalized drift condition, with respect to the measure $m$ involved in Condition {\bf C'}, i).

\noindent
{\bf Condition D.} For all $r>0$ there exists $N_0 > 0$ such that
$$
\sup\limits_{n\geq N_0}m(1_{[V\leq r]}S_n(b^2)) < \infty.
$$

\begin{prop} \label{prop 2}
Let $m$ be a non-trivial finite measure.
Assume that the generalized drift condition and Condition {\bf D} hold.

Then Assumption {\bf C'}, ii) is satisfied w.r.t $m$.
\end{prop}

\begin{proof}
As in the beginning of the proof for Proposition \ref{prop 1}, iii), and by Cauchy-Schwartz inequality, we obtain that
$$
S_n(b^2)^{\frac{1}{2}}S_n(1_C)^{\frac{1}{2}} \geq S_n(b1_C)\geq 1-\frac{V}{n},
$$ hence
$$
1_{[V\leq n_0]}S_n(b^2)^{\frac{1}{2}}S_n(1_C)^{\frac{1}{2}} \geq \dfrac{1}{2b}1_{[V\leq n_0]}
$$
for all $n\geq 2n_0$, where $n_0$ is such that $m([V \leq n_0]) \geq \epsilon > 0$.
By applying one more time the Cauchy-Schwartz inequality w.r.t. $m$ from this time, it follows that
$$
m(S_n(1_C)) \geq \dfrac{\epsilon^2}{4b^2m(1_{[V\leq n_0]}S_n(b^2))}
$$
for all $n\geq 2n_0$.

The result now follows due to the hypotheses.
\end{proof}

The announced answer to Tweedie's question is now a collection of the above results. To make it more clear, we consider:

\vspace{0.2cm}
\noindent
{\bf Condition E.} Assume that the generalized drift condition, Condition {\bf D}, and Assumption {\bf C'}, i) are verified.

\begin{coro} \label{coro 3.10}
If Condition {\bf E} is satisfied then $m \circ R$ is almost invariant.
\end{coro}

\begin{proof}
By the hypothesis and Proposition \ref{prop 2} we have that Condition {\bf C'}  is verified.
Now, the result follows by Theorem \ref{thm 3.8.0}.
\end{proof}

Recall that a set $A \in \mathcal{B}$ is called {\it absorbing} if $P(A,x) = 1$ on $A$.
In probabilistic terms, this means that if the process starts from $A$ it remains in $A$.

\begin{coro} \label{coro 3.11}
Let $E$ be a universally measurable separable metric space. 
Consider that Assumption {\bf C'}, i) (and hence ii)) holds for $C=E$, and the function $\phi$ has an increasing inverse.
Then $m\circ R$ is mean almost invariant and the number of all orthogonal invariant probability measures is less than $\dfrac{m(E)}{\phi^{-1}(1-\delta)}$.
Consequently, if $\phi(\frac{m(E)}{2})<1-\delta$ then there is a unique invariant measure (hence ergodic).
\end{coro} 

\begin{proof} 
The fact that $m\circ R$ is mean almost invariant follows by Theorem \ref{thm 3.8.0}.
Using e.g. \cite{BeCiRo15}, Proposition 2.4, one can show that the support of an invariant probability measure contains an absorbing set of total mass equal to $1$. 
But if $A\in \mathcal{B} $ is absorbing and $x \in A$,
 then  $1=P1_A(x) \leq \phi(m(A)) + \delta$, hence $m(A) \geq \phi^{-1}(1-\delta)$ and the proof for the first assertion follows.
Now, clearly $\phi(\frac{m(E)}{2})<1-\delta$ is the same with $\dfrac{m(E)}{\phi^{-1}(1-\delta)}<2$, hence there can not be two orthogonal invariant measures.
On the other hand (cf. e.g. \cite{BeCiRo15}, Proposition 4.4), any two distinc extremal invariant probability measures are singular.
This means that there is a unique extremal invariant probability measure.
The uniqueness of the invariant probability measure follows by the fact that all invariant probability measures can be represented by means of the extremal ones; see e.g. \cite{Ma77}. 

\end{proof}

\subsection{Harnack type inequalities and almost invariant measures}

\noindent
{\bf Applications to nonlinear SPDEs.} 
Let $V \subset H \equiv H^{\ast} \subset V^{\ast}$ be a Gelfand triple, i.e. $(V, \| \cdot \|_V)$ is a reflexive Banach space which is continuously and densely embedded in a separable Hilbert space $(H, \langle \cdot, \cdot \rangle, \| \cdot \|_H)$. 
The duality between $V^{\ast}$ and $V$ is denoted by ${}_{V^{\ast}}\langle \cdot, \cdot \rangle_V$. 
Let $(L_2(H), \| \cdot \|_2)$ denote the Hilbert space of all Hilbert-Schmidt operators on $H$, with the associated norm.

Let $(W(t))_{t \geq 0}$ be the cylindrical Brownian motion on $H$ w.r.t. a complete filtered space $(\Omega, \mathcal{F}, (\mathcal{F}_t)_{t \geq 0}, P)$.
Consider the following nonlinear equation with additive noise
$$
dX(t) = A(X(t)) dt + B dW(t), \leqno(3.5)
$$
where $A : V \to V^{\ast}$ and $B \in L_2(H)$ satisfy the following conditions:

\vspace{0.2cm}
\noindent
($H_1$) (Hemicontinuity) For all $u, v, x \in V$ the map
$$
\mathbb{R} \ni \lambda \longmapsto {}_{V^{\ast}}\langle A(u +\lambda v), x \rangle_V
$$
is continuous.

\vspace{0.2cm}
\noindent
($H_2$) (Weak monotonicity) There exists $c \in \mathbb{R}$ such that for all $u, v \in V$
$$
2{}_{V^{\ast}}\langle A(u) - A(v), u-v \rangle_V \leq c \| u-v \|^2_H.
$$
\noindent
($H_3$) (Coercivity) There exist $\alpha \in (0, \infty)$, $c_1 \in \mathbb{R}$, $f, c_2 \in (0, \infty)$ such that
$$
2{}_{V^{\ast}}\langle A(v), v \rangle_V + \| B \|^2_2 \leq c_1 \| v \|^2_H - c_2\| v \|^{\alpha +1}_V + f.
$$
\noindent
($H_4$) (Growth) For all $u, v \in V$
$$
| {}_{V^{\ast}}\langle A(v), u \rangle_V | \leq f + c_1(\|v\|^{\alpha}_V + \|u\|^{\alpha + 1}_V + \|u\|^2_H + \|v\|^2_H).
$$

\vspace{0.3cm}

By \cite{KrRo79} (see also \cite{LiRo10}) there exists a strong solution for equation (3.5), i.e. 
there exists a continuous $H$ -valued adapted process $X = X(t)_{t \geq 0}$ s.t.
$$
X(t) = X(0) + \int_0^t A(X_s)ds + B(W(t))
$$
and
$$
E\left(\int_0^t \|X(s)\|^{\alpha + 1}_V + \| X(s) \|^2_H ds\right) < \infty, t > 0
$$
for every $X(0) \in L^2(\Omega, \mathcal{F}_0, P; H)$.

Moreover $(X(t))_{t \geq 0}$ is a time-homogeneous Markov processes on $H$ with transition function 
$P_tf(x) := E(f(X^x(t))), \; f \in p\mathcal{B}(H)$, $x\in H$, where $X^x(t)$ is the solution of (3.5) with $X^x(0) = x$.

Our aim is to investigate the existence of invariant measures for $(P_t)_{t \geq 0}$ defined above, using our two step-approach. 
To do this, let us consider the following assumptions.

\noindent
{\bf Assumption F}. (Strict coercivity w.r.t. $\|\; \|_H$) There exist $\beta, g \in (0, \infty)$ such that
$$
2{}_{V^{\ast}}\langle A(v), v \rangle_V + \|B\|^2_2 \leq -\beta\|v\|^2_H + g
$$
for all $v \in V$.

\noindent
{\bf Assumption G}. There exists $p \geq 1$ such that for all $t > 0$ and for every ball $B_H(0, R)$ 
of radius $R$ there exists a constant $a_t(R) < \infty$ s.t. for all $x, y \in B_H(0, R)$ and $f \in p\mathcal{B}(H)$
$$
(P_tf(y))^p \leq a_t(R) \cdot P_t(f^p)(x).
$$ 

\vspace{0.2cm}

Assumption {\bf G} is a generalization of the famous Wang's Harnack inequality \cite{Wa13}.

It is well known that if ${\rm dim}H < \infty$, then Assumption {\bf F} 
ensures the existence of an invariant probability measure for $(P_t)_{t \geq 0}$; see \cite{PrRo07}, Proposition 4.3.5.
If ${\rm dim}H = \infty$ and the embedding $V \subset H$ is compact, 
then under a strict coercivity condition w.r.t. $\|\; \|_V$, namely
$$
2{}_{V^{\ast}}\langle A(v), v \rangle_V + \|B\|^2_2 \leq -\beta\|v\|^{1+\alpha}_V + g
$$
for all $v \in V$, the existence of an invariant probability measure is still guaranteed, as shown in \cite{Wa13}, Proposition 2.2.3.
Clearly, since $\|\; \|_V$ is stronger than $\|\;\|_H$, the above inequality is more restrictive than Assumption {\bf F}.
As a matter of fact, Assumption {\bf F} is considered because it guarantees that the solution $X$ is bounded in probability, i.e. $\lim\limits_{R\to \infty} \sup\limits_{t \geq 0}P(\|X_t\|_H \geq R) = 0$ (hence the existence of an invariant probability measure if ${\rm dim} H<\infty$). 
As noted in \cite{DaZa96}, in general, the boundness in probability property is not sufficient to ensure the existence of an invariant measure for $(P_t)_{t \geq 0}$ even for deterministic equations, and we refer to \cite{Vr93} for a counterexample.  
However, we can show that Assumption {\bf F} in combination with Assumption {\bf G} does imply the existence of an invariant measure. 
To the best of our knowledge, this result is new in the literature and we present it below (the uniqueness and full support properties were already known, see \cite{Wa13}, Theorem 1.4.1 and Corollary 2.2.4). 

Recall that by Theorem \ref{thm 3.8.0}, the following condition ensures the existence of an invariant probability measure for a Markov kernel $P$ on a measurable space $(E, \mathcal{B})$:

\vspace{0.2cm}
\noindent
{\bf Assumption C'}. i) There exist a finite positive measure $\nu$ on $E$, 
a nonempty set $C \in \mathcal{B}$, a function $\phi:[0, \infty) \rightarrow [0, \infty)$ which is continuous and zero in $0$, and $\delta \in [0,1)$ such that
$$
Pf(x) \leq \phi(\nu(f)) + \delta
$$
for all $f \in \mathcal{B}_1^+$ and $x \in C$.

ii) There exists $n_0 > 0$ such that $\mathop{\inf}\limits_{n \geq n_0} \nu(S_n 1_C)>0$, where $C$ is the set from i) and
$S_n := \dfrac{1}{n} \mathop{\sum}\limits_{k = 0}^{n - 1} P^k$.

\begin{prop} \label{prop 3.12}
Suppose that assumptions {\bf F} and {\bf G} are satisfied. 
Then there exists a unique invariant probability measure for $(P_t)_{t \geq 0}$ and it has full support on $H$.
\end{prop}

\begin{proof}
First, note that the strict coercivity assumption implies that $E(\|X_t^0\|^2_H) \leq \dfrac{g}{\beta}$ for all $t \geq 0$ (cf. [R\"o Pr], pg. 103, (4.3.12)), hence $P_t(1_{B_H(0, R)})(0)$ $\geq 1 - \dfrac{g}{\beta R^2}$, $t > 0, \; R > 0$. 
Now fix $R$ large enough so that $\mathop{\inf}\limits_{t > 0} P_t(1_{B_H(0, R)})(0) > 0$ and recall that by Assumption {\bf G} $(P_tf(y))^p \leq a_t(R)P_t(f^p)(x)$ for all $x, y \in B_H(0, R),\; f\in p\mathcal{B}(H)$.

If we fix $t>0$ and set $\nu(f) := \delta_0 \circ P_t(f)$, $f \in p\mathcal{B}$, and $\phi(x) = \sqrt[p]{a_t(R)x}$, $x \geq 0$, we obtain that $P_tf(y) \leq \phi(\nu(f))$ for all $y \in B_H(0, R)$, $f \in \mathcal{B}_1^+$, and $\mathop{\inf}\limits_{n} \nu(S_n1_{B_H(0, R)}) > 0$. 
Therefore, Assumption {\bf C'} is satisfied for the Markovian kernel $P_t$, and by Theorem \ref{thm 3.8.0} and Proposition \ref{prop 2.3}, the existence part is proved.

As we already mentioned, the uniqueness and the fact that the invariant measure has full support follow in a similar way as for Corollary 2.2.4 from \cite{Wa13}.

\end{proof}

For the reader's convenience we recall some sufficient conditions under which Assumption {\bf G} is satisfied.

Assume that the operator $B$ is non-degenerate, i.e. if $v \in H$ and $Bv = 0$ then $v = 0$, and denote by $\|\cdot\|_B$ the intrinsic norm induced by $B$ defined as
$$
\|v\|_B = \left\{ \begin{array}{l}
\|y\|_H, \quad {\rm if \; there \; exists \;} y \in H \; {\rm such \; that} \; By = v\\[3mm] 
\infty, \quad {\rm otherwise}.
\end{array}\right.
$$

Consider the following assumptions (cf. \cite{Wa13}, (2.5), or \cite{Li09}, (1.3) and (1.4)):

\vspace{0.2cm}

\noindent
($G_1$) $\alpha \geq 1$ and there exist $\theta \in [2, \infty) \cap (\alpha -1, \infty)$ and $\eta, \gamma \in \mathbb{R}$ with $\eta > 0$ such that for all $u, v \in V$
$$
2{}_{V^{\ast}}\langle A(u) - A(v), u-v \rangle \leq -\eta \|u-v\|_B^{\theta}\|u-v\|^{\alpha + 1 - \theta}_H + \gamma\|u-v\|_H^2.
$$

\noindent
($G_2$) $\alpha \in (0, 1)$ and there exist some measurable function $h : V \to (0, \infty)$, some constant $\theta \geq \dfrac{4}{\alpha + 1}$, some $\gamma \in \mathbb{R}$, and some strictly positive constants $q, \delta, \eta$ such that for all $u, v \in V$
$$
2{}_{V^{\ast}}\langle A(u), u \rangle_V + \|B\|^2_2 \leq q - \delta h(u)^{\alpha + 1} + \gamma\|u\|^2_H,
$$
$$
2{}_{V^{\ast}}\langle A(u) - A(v), u-v \rangle_V \leq -\dfrac{\eta\|u-v\|^{\theta}_B}{\|u-v\|^{\theta - 2}_H(h(u)\vee h(v))^{1 - \alpha}} + \gamma\|u-v\|^2_H.
$$

\vspace{0.2cm}

By \cite{Wa13}, Theorems 2.2.1 and 2.3.1, the following two assertions concerning Harnack inequalities hold:

\vspace{0.2cm}

a). If ($G_1$) holds then for every $p > 1, t > 0, x,y \in H$ and $f \in p\mathcal{B}(H)$
$$
(P_tf(y))^p \leq P_tf^p(x)\exp\left[\frac{p(\frac{\theta + 2}{\theta + 1 - \alpha})^{\frac{2(\theta + 1)}{\theta}}\|x-y\|^{\frac{2(\theta + 1 - \alpha)}{\theta}}_H}{2(p - 1)(\int_0^t \eta^{\frac{2}{\theta + 2}}e^{-\frac{\theta + 1 - \alpha}{\theta + 2}\gamma s}ds)^{\frac{\theta + 2}{\theta}}}\right]. \leqno(3.6)
$$

\vspace{0.2cm}

b). If ($G_2$) holds then there exists a constant $c > 0$ s.t. for all $t > 0, p > 1, x,y \in H$, and $f \in p\mathcal{B}(H)$
$$
\hspace*{-90mm}{}(P_tf(y))^p \leq \leqno(3.7)
$$
$$P_tf^p(x)\exp\left[\frac{cp}{p-1}\left(\frac{\|x-y\|^2_H(1 + \|x\|^2_H + \|y\|^2_H)}{(t \wedge 1)^{\frac{\theta(\alpha + 1) + 4}{\theta(\alpha + 1)}}} + (\frac{p}{p-1})^{\frac{4(1-\alpha)}{\alpha(\theta +2) + \theta - 2}}\frac{\|x-y\|^{\frac{2\theta(\alpha +1)}{\alpha(\theta+2)+\theta-2}}}{(t\wedge 1)^{\frac{\theta(\alpha +1) +4}{\alpha(\theta+2)+\theta-2}}}\right)\right]. 
$$
Concrete examples of operators $A,B$ satisfying conditions ($G_1$) or ($G_2$) have been constructed in \cite{Wa13}, subsection 2.4, for  stochastic generalized porous media, $p$-Laplace, and generalized fast-diffusion equations; see also \cite{Li09}.

\begin{rem} i) The existence part of Proposition \ref{prop 3.12} may be also proved by applying Theorem \ref{thm 3.5} (due to Tweedie), since one can see that for fixed $P_t$, Assumption {\bf B}, i) is satisfied for $v = \|\cdot\|^2_H$ and $C = B_H(0, R)$ with $R$ sufficiently big, while Assumption {\bf B}, ii) follows by the Harnack inequality.

\vspace{0.2cm}

ii) We would like to point out that we considered only the additive noise case just because in this situation it is already known that if ($G_1$) or ($G_2$) hold then inequalities (3.6) resp. (3.7) are satisfied (hence so is Assumption {\bf G}). 
In fact, one can easily see that Proposition \ref{prop 3.12} is true also for the multiplicative case, i.e. $B : V \to L_2(H)$.

\end{rem}

\noindent
{\bf Perturbations of Markov chains satisfying a combined Harnack-Lyapunov condition.}
Let $P$ be a Markov kernel on $(E, \mathcal{B})$ satisfying the following condition:

\vspace{0.2cm}
\noindent
(H-L). There exists a positive measurable function $V$ such that $PV \leq \gamma V + c$ for some positive constants $c$ and $\gamma < 1$.
Moreover, for each $r > 0$ there exist a point $z_0 \in E$, $p > 1$, and a constant $M = M(r, z_0)$ such that $(Pf(x))^p \leq M Pf^p(z_0)$ for all $x \in [V \leq r]$.

\vspace{0.2cm}

Clearly, an example of such a kernel is any $P_t$ associated to the previous SPDE, under Assumptions {\bf F} and $ G_1$ (or $G_2$).

Let now $\rho : E \to (0,1)$ be a measurable function such that $0 < a := \mathop{\inf}\limits_{x \in E} \rho(x)$ and $\mathop{\sup}\limits_{x \in E} \rho(x) = :b < 1$, and $Q$ be a second Markovian kernel on $(E, \mathcal{B})$.

In the sequel we are interested in showing the existence of an invariant probability measure not for $P$ (for which we already now that such a measure exists; cf. Theorem \ref{thm 3.5} or by our generalization Theorem \ref{thm 3.8.0}) but for the following modified Markov kernel
$$
\overline{P}f := \rho Pf + (1-\rho)Qf, \; f \in p\mathcal{B}.
$$

\vspace{0.3cm}

\begin{coro} \label{coro 3.14}
Assume there exist constants $\eta$ and $l < \frac{1-b\gamma}{1-a}$ such that $QV \leq lV + \eta$. 
Then there exists an invariant probability measure for $\overline{P}$.
\end{coro}

\begin{proof} By hypothesis,
$$
\overline{P}f \leq bPf + 1 - a \leq b\sqrt[p]{MPf(z_0)} + 1 - a \leq b\sqrt[p]{\frac{M}{a}\overline{P}f(z_0)} + 1 - a
$$
for all $f \in \mathcal{B}_1^+$ and $x \in [V \leq r]$.
Hence, if $m := \delta_{z_0} \circ \overline{P}$ and $\phi(t) = b\sqrt[p]{\dfrac{M}{a}t}$, $t \geq 0$, 
we obtain that $\overline{P}$ satisfies Assumption {\bf C'}, i) for $C = [V \leq r]$. 
On the other hand, $V$ is a Lyapunov function for $\overline{P}$ too because
$$
\overline{P}V \leq \rho\gamma V + (1 - \rho)lV + bc + (1-a)\eta \leq (b\gamma +(1-a)l))V + bc + (1-a)\eta,
$$
and $b\gamma +(1-a)l < 1$.
But as in the proof of Proposition \ref{prop 1}, i) (relation ($\ast$)) we obtain that 
$\mathop{\inf}\limits_k \overline{P}^k1_{[V \leq r]}(z_0) > 0$ 
which leads to $\mathop{\inf}\limits_k m(\overline{S}_k 1_{[V \leq r]}) > 0$ for sufficiently large $r$. 
Consequently, $\overline{P}$ satisfies Assumption {\bf C'} and the result follows by Theorem \ref{thm 3.8.0}.
\end{proof}

In the particular case when $Q(x,\cdot)= \delta_x$, $x \in E$, let us denote $\overline{P}$ by $P^{\rho}$, i.e.
$$
P^{\rho}f(x) = \rho(x)Pf(x) + (1 -\rho(x))\delta_x(f)
$$
for all $f \in p\mathcal{B}$.

By Corollary \ref{coro 3.14} we get:

\begin{coro}
The Markov kernel $P^{\rho}$ admits an invariant probability measure.
\end{coro}

For the reader's convenience, we recall the assumptions involved in the results of Harris and Tweedie, 
respectively, which ensure the existence of an invariant probability measure for $P$, 
as we already discussed in Subsection 3.2 (see Theorems \ref {thm 3.4} and \ref{thm 3.5}): 

\vspace{0.2cm}
\noindent
{\bf Assumption A.} 
There exist a function $\widetilde{V} \in p\mathcal{B}$ and constants $\widetilde{b}$ and $\widetilde{\gamma} \in (0, 1)$ such that
$$
P\widetilde{V} \leq \widetilde{\gamma} \widetilde{V} + \widetilde{b} \;\;\; {\rm on} \; E.
$$  

\noindent
Furthermore, the sub-level set $[\widetilde{V} \leq r]$ is small for some $r > 2\widetilde{b} \slash (1-\widetilde{\gamma})$, i.e.

$$
\inf\limits_{x \in [\widetilde{V} \leq r]}P(x, \cdot) \geq \nu(\cdot)
$$
for some non-zero sub-probability $\nu$.

\vspace{0.2cm}
\noindent
{\bf Assumption B.} i) There exist a measurable function $\widetilde{V}:E \rightarrow [0, \infty)$, a constant $\widetilde{b}$, and a set $C \in \mathcal{B}$ such that
$$
P\widetilde{V} \leq \widetilde{V} -1 + \widetilde{b}1_C \;\;\; {\rm on} \; E.
$$

ii) The set $C$ from i) is such that the following uniform countable additivity condition holds: 
for all $(A_n)_n \subset \mathcal{B}$ decreasing to $\emptyset$ we have that
$$
\lim\limits_n\sup\limits_{x\in C}P(A_n)(x)=0.
$$

Our next purpose is to show that $P^\rho$ does not satisfy Assumption {\bf B}. 
Regarding the aplicability of Harris' result we do not have a similar negative answer. 
However, we can show that if $P$ does not satisfy Assumption {\bf A} then neither does $P^{\rho}$.
More precisely, we have:

\begin{prop} \label{prop 3.16}
The following assertions hold for the kernel $P^{\rho}$:

\vspace{0.2cm}

i) If a non-empty set $C \in \mathcal{B}$ satisfies Assumption {\bf B}, ii) then $C$ 
consists of a finite number of points. In particular, if $P(x, \{y\}) = 0$ for all $x, y \in E$ then $P^{\rho}$ does not satisfy Assumption {\bf B}.

\vspace{0.2cm}

ii) If $P$ does not satisfy Assumption {\bf A} and $P(x, \{y\}) = 0$ for all $x, y \in E$ then $P^{\rho}$ 
does not satisfy Assumption {\bf A} provided $a > \frac{1}{2}$.

\end{prop}

\begin{proof}

i). Assume that the set $C$ is not finite, so we can find a sequence 
$(A_n)_n \subset \mathcal{B}$, $\emptyset \neq A_n \subset C$, decreasing to $\emptyset$. 
Then $\mathop{\sup}\limits_{x \in C} P^{\rho}(x, A_n) \geq 
\mathop{\sup}\limits_{x \in C} (1 -\rho(x))\delta_x(A_n) \geq (1 -b)\mathop{\sup}\limits_{x \in C}\delta_x(A_n) = 1-b$ 
for all $n \geq 1$, hence Assumption {\bf B}, ii) is not verified, which is a contradiction.

Suppose now that $P(x, \{y\}) = 0$ for all $x,y \in E$. We claim that $(P^{\rho})^n(x, C) \mathop{\to}\limits_{n}0$ 
for any $C = \{ x_1, \ldots, x_n \} \subset E$ and $x \in E$. 
Clearly, it is enough to show this for $C = \{y\}$ for some arbitrarily fixed $y \in E$. 
But 
$$
P^{\rho}(x, \{y\}) = \rho(x)P(x, \{y\}) + (1-\rho(x))\delta_x(\{y\}) = (1-\rho(y))1_{\{y\}}(x) \; {\rm for \; all} \; x \in E.
$$
$$
(P^{\rho})^{2}(x, \{y\}) = (1-\rho(x))(1-\rho(y))P(x,\{y\}) + (1-\rho(x))(1-\rho(y))1_{\{y\}}(x) = 
$$
$$
= (1-\rho(y))^21_{\{y\}}(x) \; {\rm for \; all} \; x \in E.
$$
Inductively, one gets
$$
(P^{\rho})^{n}(x,\{y\}) = (1-\rho(y))^n1_{\{y\}}(x) \leq (1-a)^n \mathop{\to}\limits_{n}0.
$$
Now, if $P^{\rho}$ satisfies Assumption {\bf B} for some set $C \in \mathcal{B}$, 
then by the above considerations $C$ must be finite and $(P^{\rho})^{n}(x, C)\mathop{\to}\limits_{n}0$ for all $x \in E$. 
But this implies that $\overline{S}_n(x, C) = 
\dfrac{1}{n}\mathop{\sum}\limits_{k=0}^{n-1}(P^{\rho})^{k}(x, C)\mathop{\to}\limits_{n}0$ 
for all $x \in E$, which contradicts Proposition \ref{prop 1}, iii).

\vspace{0.2cm}

ii). Assume that $P^{\rho}$ satisfies Assumption {\bf A}, so that there exists 
$\widetilde{V} \in p\mathcal{B}$ s.t. $P^{\rho}\widetilde{V} \leq \widetilde{\gamma} \widetilde{V} + \widetilde{b}$ 
for some positive constants $b$ and 
$\gamma < 1$, and $r > \dfrac{2\widetilde{b}}{1-\widetilde{\gamma}}$ s.t. 
$\mathop{\inf}\limits_{x \in [\widetilde{V} \leq r]}P^{\rho}(x, \cdot) \geq \nu(\cdot)$ for some non-zero sub-probability $\nu$.
Then
$$
P\widetilde{V}(x) = \frac{1}{\rho(x)}P^{\rho}\widetilde{V}(x) - \frac{1-\rho(x)}{\rho(x)}\widetilde{V}(x)
\leq \frac{\widetilde{\gamma} - 1 + \rho(x)}{\rho(x)} \widetilde{V}(x) + \frac{\widetilde{b}}{\rho(x)}
\leq \widetilde{\gamma} \widetilde{V}(x) + \frac{\widetilde{b}}{a},
$$
for all $ x \in E$, therefore $\widetilde{V}$ is a Lyapunov function for $P$.

The next step is to prove that $[\widetilde{V} \leq r]$ is small for $P$, which clearly completes the proof, 
since it would contradict the hypothesis that $P$ does not satisfy Assumption {\bf A}. 
To this end, let us notice first that $[\widetilde{V} \leq r]$ is uncountable, in particular $[\widetilde{V} \leq r]$ contains at least two points. 
Indeed, as in the proof of Proposition \ref{prop 1}, i), we have that $P^n([\widetilde{V} > r]) \leq 
\dfrac{\widetilde{\gamma}^n\widetilde{V}}{r} + \dfrac{1}{2a}$ for all $n \geq 1$, 
so that $P^n(x, [\widetilde{V} \leq r]) \geq \varepsilon$ for all $n$ 
large enough, some arbitrarily fixed $x$, and some $\varepsilon = \varepsilon(x) > 0$. 
If $[\widetilde{V} \leq r]$ is countable, then $P^n(x, [\widetilde{V} \leq r]) = 0$ 
(since $P(x, \cdot)$ does not charge the points) which is a contradiction.

Let now $y \in E$ arbitrarily chosen and let $x \in [\widetilde{V} \leq r], x \neq y$.
Then $\nu(\{y\}) \leq P^{\rho}(x, \{y\}) = \rho(x)P(x, \{y\}) + (1-\rho(x))\delta_x\{y\} = 0$, hence $\nu$ does not charge the points as well. 
Then, for $A \in \mathcal{B}$
$$
P(x, A) = P(x, A\setminus\{x\}) = \dfrac{1}{\rho(x)}P^{\rho}(x, A\setminus\{x\}) - \dfrac{1-\rho(x)}{\rho(x)}\delta_x(A\setminus \{x\}) =
$$
$$
= \dfrac{1}{\rho(x)}P^{\rho}(x, A\setminus\{x\}) \geq \dfrac{1}{b}\nu(A\setminus\{x\}) = \dfrac{1}{b}\nu(A) \; {\rm for \; all} \; x \in [\widetilde{V} \leq r].
$$
Consequently, $[\widetilde{V} \leq r]$ is small for $P$.

\end{proof}

\begin{rem}
i) It is worth to mention that with small changes in the proof, 
Proposition \ref{prop 3.16}, i) remains true for more general $\overline{P}$ when $Q$ is not necessarily the identity kernel. 
However, we assume that $Q$ inherits the following property: 
there exists $\varepsilon>0$ such that $Q(x, \{x\}) > \varepsilon$ and $Q(x, \{y\})=0$ for all $x\neq y \in E$.  

ii) As already mentioned in Remark \ref{rem 3.8}, Assumption {\bf A}
is often formulated such that $[\widetilde{V} \leq r]$ is small for all sufficiently large $r>0$. 
In this situation, one can easily see that with essentially the same proof, 
no lower bound for the constant $a$ is needed in order for Proposition \ref{prop 3.16}, ii) to be true. 
\end{rem}

\subsection{Uniform boundness on $L^1$ implies uniform integrability for the adjoint} 

From now on we assume that $(P_t)_{t \geq 0}$ is a strongly continuous semigroup of Markovian operators on $L^p(m)$ for some $p \geq 1$. 
We consider the associated resolvent $(R_{\alpha})_{\alpha > \alpha_0 \geq 0}$
$$
R_{\alpha}(f) = \int_{0}^{\infty}e^{-\alpha t}P_tf dt, \; f \in L^p(m).
$$

\begin{rem} \label{rem 3.2}
i) In general, $R_{\alpha}$ is not defined on $L^p(m)$ for all $\alpha > 0$, unless $(0, \infty)$ is included in the resolvent set of the generator associated with $(P_t)_{t \geq 0}$.
However, $R_\alpha$ can be defined on $L^p(m)\cap L^\infty(m)$, for all $\alpha > 0$.

ii) If $(P_t)_{t \geq 0}$ is uniformly bounded then $(0, \infty)$ 
is included in the resolvent set of the generator and $(\alpha R_\alpha)_{\alpha > 0}$ 
is uniformly bounded, but the converse is not necessarily true in general.
\end{rem}

\begin{thm} \label{prop 3.3}
Assume that $R_{\alpha}$ is defined for all $\alpha > 0$ and that the family $(\alpha R_{\alpha})_{\alpha > 0}$ is uniformly bounded on $L^p(m)$. Then $m$ is resolvent almost invariant, hence there exists a nonzero positive finite invariant measure $\nu = \rho \cdot m$. 
Moreover, if $p > 1$ then $\rho$ can be chosen from $L^q_{+}(m)$.
\end{thm}

\begin{proof}
The first part follows since for all $A \in \mathcal{B}$
$$
m (\alpha_n R_{\alpha_n}1_A) \leq m(E)^{\frac{p - 1}{p}}\cdot M \cdot m(A)^{\frac{1}{p}}
$$

If $p>1$, since $(\alpha R_\alpha^\ast 1)_{\alpha > 0}$ is bounded in $L^q(m)$ 
and arguing as in the proof of Theorem \ref{thm 2.3}, v)$\Rightarrow$ i), 
one can see that any accumulation point $\rho$ of $(\alpha R_\alpha^\ast 1)_{\alpha > 0}$ in $L^q(m)$ 
leads to a non-zero finite invariant measure $\rho \cdot m$.
\end{proof}

\vspace{0.2cm}

Extending \cite{Hi00}, a positivity preserving operator $P$ on $L^p(m)$, $p \geq 1$ is said to satisfy {\it condition (I)}
if there exists $\phi \in L^p_+(m)$ such that 
$$
\mathop{\lim}\limits_{r \to \infty} \mathop{\sup}\limits_{f \in L^p(m), \|f\|_p \leq 1} \|(|Pf|- r\phi)^+\|_p < 1.
$$ 

Closely related to condition (I), the following $L^p$-{\it tail} norm was considered in \cite{Wa05} 
in order to measure the non(semi)compactness of a bounded operator 
$P$ on $L^p(m)$, for a fixed $\phi \in L^p_+(m)$, $p \geq 1$:
$$
\| P \|_{p, T}^{\phi} = \mathop{\lim}\limits_{r \to \infty} \mathop{\sup}\limits_{\| f \|_p \leq 1} \| (Pf)1_{\{ |Pf| > r\phi \}} \|_p.
$$

\begin{rem} \label{rem 3.4}
i) If a positivity preserving operator $P$ on $L^p(m)$ satisfies $\| P \|_{p, T}^{\phi} < 1$ 
then it satisfies condition (I), since $\|(|Pf|-r\phi)^+\|_p = \|(|Pf|-r\phi)^+ 1_{[|Pf|>r\phi]}\|_p \leq \|(Pf) 1_{[|Pf|>r\phi]}\|_p$.

ii) Hino used condition (I) (for $\phi = 1$) in order to show the existence of a nonzero element $\rho$ in $\ker P^{\ast}$ 
(hence a nonzero invariant measure for $P$). 
More precisely, he showed that if $P$ is a Markovian operator on a separable $L^p(m)$ space with $p>1$ 
such that $P^n$ satisfies condition (I) for some $n > 0$ and $\phi = 1$, then there exists a nonzero element $\rho \in \ker (I - P^{\ast})$. 
Then he applied this result for $P=P_{t_0}$, for some $t_0 > 0$, 
and $P=\alpha R_{\alpha}$, for some $\alpha > 0$; see [Hi 00], Theorems 2.8 and 2.9.
\end{rem}

We recapture Hino's results in the following more general statement, 
as a particular case of Theorem \ref{prop 3.3}. 
Our main improvements consists of allowing $p=1$ and remaining in the case of an arbitrary measurable space $(E, \mathcal{B})$. 

\begin{coro} \label{coro 3.5}
i) If there exist $t_0 > 0$ and $\phi \in L^p(m)$ such that $P_{t_0} \phi \leq \phi$ and $P_{t_0}$ 
satisfy  condition (I) then $(P_t)_{t \geq 0}$ is uniformly bounded.

ii) If there exist $n \geq 1$, $\alpha > \alpha_0$, and $\phi \in L^p(m)$ such that $(\alpha R_{\alpha})^n \phi \leq \phi$ and $(\alpha R_{\alpha})^n$ satisfies condition (I), then $R_{\beta}$ is defined for all $\beta > 0$ and $(\beta R_{\beta})_{\beta > 0}$ is uniformly bounded. 

In particular, if the assumptions in i) or ii) are satisfied then the conclusion of Theorem \ref{prop 3.3} holds.
\end{coro}

\begin{proof} 
By a simple adaptation of the proof for Proposition 2.5 in \cite{Hi00}, one can show that under i) it follows that $(P_t)_{t > 0}$ is uniformly bounded, and respectively, in the case of ii), that $\{ (\alpha R_{\alpha})^k \}_{k \geq 1}$ is uniformly bounded. 
So, let $M< \infty$ be a positive real number such that $\| (\alpha R_{\alpha})^k \|_{L^p(m)} \leq M$ for all $k > 0$. 
For $0 < \beta < \alpha$, define
$$
\widetilde{R}_{\beta} := \mathop{\sum}\limits_{k = 0}^{\infty} (\alpha - \beta)R_{\alpha}^{k + 1}.
$$

Then 
$$
\| \widetilde{R}_{\beta} \| \leq \mathop{\sum}\limits_{k = 0}^{\infty} (\alpha - \beta)^n \| R_{\alpha}^{k + 1} \| \leq \dfrac{M}{\alpha} \mathop{\sum}\limits_{k = 0}^{\infty} \left(\dfrac{\alpha - \beta}{\alpha}\right)^k = \dfrac{M}{\beta} 
$$
and it is straightforward to check that $(R_{\beta})_{\beta \geq \alpha_0}$ extends to a Markovian resolvent $(R_{\beta})_{\beta > 0}$ by setting $R_{\beta} = \widetilde{R}_{\beta}$ for all $0 < \beta < \alpha_0$, such that $(\beta R_{\beta})_{\beta > 0}$ is uniformly bounded.
\end{proof}

{\bf Applications to sectorial forms.} Let $(\mathcal{E}, D(\mathcal{E}))$ 
be a coercive closed form on $L^2(m)$ in the sense of \cite{MaRo92}. 
That is, $D(\mathcal{E})$ is a dense linear subspace of $L^2(m)$ and 
$\mathcal{E} : D(\mathcal{E}) \times D(\mathcal{E}) \to \mathbb{R}$ is bilinear, 
nonnegative definite such that $D(\mathcal{E})$ is complete with respect to the norm 
$\| \cdot \|_{\mathcal{E}_1} := \mathcal{E}_1(\cdot, \cdot)^{\frac{1}{2}}$, 
where, for $\alpha \in \mathbb{R}_+$, $\mathcal{E}_{\alpha}(f, g) 
:= \mathcal{E}(f, g) + \alpha \langle f, g \rangle_{L^2(m)}$ for all $f, g \in D(\mathcal{E})$.

Also, the "weak sector condition" is assumed, i.e. 
there exists $k \in \mathbb{R}_+$ such that $|\mathcal{E}(f, g)| \leq k\mathcal{E}_1(f, f)^{\frac{1}{2}}\mathcal{E}_1(g, g)^{\frac{1}{2}}$ 
for all $f, g \in D(\mathcal{E})$.

Then (cf. \cite{MaRo92}, Chapter 1, Sections 1 and 2), one can associate a strongly continuous semigroup of contractions $(T_t)_{t \geq 0}$ on $L^2(m)$ whose generator $(L, D(L))$ satisfies $D(L) \subset D(\mathcal{E})$ 
densely and $\mathcal{E}(f, g) = (-Lf, g)$ for all $f \in D(L)$ and $g \in D(\mathcal{E})$.

A bilinear form $(\mathcal{E}, D(\mathcal{E}))$ on $L^2(m)$ 
is called a sectorial form if there exists $\alpha \in [0, \infty)$ 
such that $(\mathcal{E}_{\alpha}, D(\mathcal{E}_{\alpha}) := D(\mathcal{E}))$ is a coercive closed form. 
If $(T_t^{(\alpha)})_{t > 0}$ is the semigroup corresponding to $\mathcal{E}_{\alpha}$, 
then $T_t := e^{\alpha t}T_t^{(\alpha)}$, $t > 0$ is called the semigroup on $L^2(m)$ associated with $(\mathcal{E}, D(\mathcal{E}))$. 
We say that $(\mathcal{E}, D(\mathcal{E}))$ is positivity preserving if $(T_t)_{t > 0}$ is positivity preserving.

Since the semigroup generated by a coercive closed form is of contractions, the following result is an immediate consequence of Theorem \ref{prop 3.3}.

\begin{coro} \label{coro 3.6}
Assume that $(T_t)_{t > 0}$ is a Markovian semigroup associated to a coercive closed form $(\mathcal{E}, D(\mathcal{E}))$. 
Then there exists $0 \neq \rho \in L^2_+(m)$ such that $\rho \cdot m$ is $T_t$-invariant.
\end{coro}

As in [Wa 05], let us consider the following inequality for $(\mathcal{E}, D(\mathcal{E}))$: 
there exist  $r_0 \geq 0$, $\beta : (r_0, \infty) \to (0, \infty)$, and a strictly positive $\phi \in L^2(m)$ such that
$$
m(f^2) \leq r\mathcal{E}(f,f) + \beta(r)m(\phi|f|)^2, \; r > r_0, \; f \in D(\mathcal{E}). \leqno(3.2)
$$

Recall that $(\mathcal{E}, D(\mathcal{E}))$ is said to satisfy the {\it super Poincar\'e inequality} if (3.2) is satisfied for $r_0 > 0$.

Also, let $F : (0, \infty) \to \mathbb{R}$ be an increasing and continuous function such that 
$\mathop{\sup}\limits_{r \in (0, 1]} |rF(r)| < \infty$ and $\mathop{\lim}\limits_{r \to \infty} F(r) = +\infty$. 
We say that $(\mathcal{E}, D(\mathcal{E}))$ satisfy the {\it $F$-Sobolev inequality} if there exist  two constants $c_1 > 0, c_2 \geq 0$ such that
$$
m(f^2F(f^2)) \leq c_1\mathcal{E}(f, f) + c_2, \; f \in D(\mathcal{E}), \; m(f^2)=1. \leqno(3.3)
$$

If $F = \log$, then (3.3) is called ({\it defective} when $c_2 \neq 0$) {\it $\log$-Sobolev inequality}.

\begin{rem} \label{rem 3.7}
By \cite{Wa05}, Theorem 3.3.1, if $(\mathcal{E}, D(\mathcal{E}))$ is a positive bilinear form on $L^2(m)$, then $F$-Sobolev inequality implies (3.2) for $\phi \equiv 1$ and $\beta(r) = c_1 F^{-1}(c_2(1 + r^{-1}))$ for some $c_1, c_2 > 0$, where $F^{-1}(r) = \inf\{ s \geq 0 : F(s) \geq r \}$.
\end{rem}

\begin{lem} \label{thm 3.8}
Let $(\mathcal{E}, D(\mathcal{E}))$ be a positivity preserving coercive closed form which satisfies (3.2). 
Assume there exists $\alpha < \dfrac{1}{r_0}$ such that $(e^{\alpha t}T_t)_{t > 0}$ is Markovian and there exists $t_0 > 0$ with $e^{\alpha t_0}T_{t_0}\phi \leq \phi$. 
Then there exists $\rho \in L^2(m)$ such that $\rho \cdot m$ is $(e^{\alpha t}T_t)_{t \geq 0}$-invariant.
\end{lem}

\begin{proof}
By \cite{Wa05}, Theorem 3.2.2, we have that $\| T_t \|_{2, T}^{\phi} \leq e^{-\frac{t}{r_0}}$ for all $t > 0$. 
Since $\| e^{\alpha t}T_t \|_{2, T}^{\phi} = e^{\alpha t}\|  T_t \|_{2, T}^{\phi} \leq e^{(\alpha - \frac{1}{r_0})t} < 1$, 
the result follows by Remark \ref{rem 3.4}, i) and Corollary 3.5, i).
\end{proof}

\begin{coro} \label{coro 3.9}
Let $(\mathcal{E}, D(\mathcal{E}))$ be a positivity preserving sectorial form such that 
$(\mathcal{E}_{\alpha}, D(\mathcal{E}))$ satisfies the $F$-Sobolev inequality for one (and hence for all) $\alpha \geq 0$.
 If $(T_t)_{t \geq 0}$ is Markovian then there exists $\rho \in L^2(m)$ such that $\rho \cdot m$ is $(T_t)_{t \geq 0}$-invariant.
\end{coro}

\begin{proof} 
It follows by Remark \ref{rem 3.7} and Lemma \ref{thm 3.8}.
\end{proof}

\vspace{0.2cm}

\noindent
{\bf Example } (small perturbation of Dirichlet forms; cf. \cite{BoRoZh00}). 
Following \cite{BoRoZh00}, let $X$ be a locally convex topological real vector space with dual 
$X^{\ast}$, $\mathcal{B} = \mathcal{B}(X)$ its Borel $\sigma$-algebra, 
and $H$ a separable Hilbert space which is continuously embedded in $X$.

For $f \in \mathcal{F}C_b^{\infty} := 
\{ \varphi(l_1, \ldots, l_m) | m \in \mathbb{N}, l_i \in X^{\ast}, \varphi \in C_b^{\infty}(\mathbb{R}^m) \}$, $x \in X$, 
define $\nabla_H f(x)$ the element in $H$ uniquely defined by
$$
\langle \nabla_H f(x), h \rangle = \dfrac{d}{ds} f(x + sh)|_{s = 0}.
$$

Let $\mu$ be a probability measure on $(X, \mathcal{B})$ such that $\mathcal{E}_{\mu}(f, g) := \int_{X} \langle \nabla_Hf, \nabla_Hg \rangle d\mu$, $f, g \in \widetilde{\mathcal{F}C_b^{\infty}}^{\mu}$ is well defined and closable on $L^2(\mu)$, where $\widetilde{\mathcal{F}C_b^{\infty}}^{\mu}$ denotes the set of all $\mu$-classes determined by $\mathcal{F}C_b^{\infty}$.

If $\mathcal{L}_s(H)$ stands for the set of all symmetric nonnegative definite bounded linear operators on $H$, 
then for any strongly measurable map $A : X \to \mathcal{L}_s(H)$ such that
$$
C^{-1}I_H \leq A \; \mu - {\rm a.e. \; for \; some \;} c > 0 {\; \rm and} \; \int_{X}\| A(x) \|_{\mathcal{L}(H)} \mu(dx) < \infty
$$
the form
$$
\mathcal{E}_{\mu, A}(f,g) := \int_{X}\langle A(x)\nabla_Hf(x), \nabla_Hg(x) \rangle_H \mu(dx), \; f,g \in \widetilde{\mathcal{F}C_b^{\infty}}^{\mu}
$$
is well defined and closable on $L^2(\mu)$, and its closure, denoted by $(\mathcal{E}, D(\mathcal{E}))$ 
for some fixed $\mu$ and $A$, is a symmetric coercive closed form.

Let $v : X \to H$ be $\mathcal{B}(X) / \mathcal{B}(H)$-measurable with $\| v \|_H \in L^2(\mu)$ 
such that there exist $\varepsilon \in (0, 1)$, $a, b \in \mathbb{R}_+$ such that for all $f,g \in \mathcal{F}C_b^{\infty}$ we have
$$
| \int \langle v, \nabla_Hf \rangle_H g d\mu | \leq b\mathcal{E}_1(f,f)^{\frac{1}{2}}\mathcal{E}_1(g,g)^{\frac{1}{2}}
$$
and
$$
\int \langle v, \nabla_Hf \rangle_H g d\mu \geq -\varepsilon \mathcal{E}(f,f) - a\|f\|^2_{L^2(\mu)}.
$$

Let $\mathcal{E}_v(f,g) := \mathcal{E}(f,g) + \int \langle v, \nabla_Hf \rangle_H g d\mu$, $f,g \in \mathcal{F}C_b^{\infty}$. 
Then $(\mathcal{E}_v, \widetilde{\mathcal{F}C_b^{\infty}})$ 
is closable and if $(\mathcal{E}_v, D(\mathcal{E}_v))$ denotes its closure, 
then $D(\mathcal{E}_v) = D(\mathcal{E})$, and for 
$\alpha := a + 1 - \varepsilon$, $(\mathcal{E}_{v, \alpha}, D(\mathcal{E}))$ is a coercive closed form (hence $\mathcal{E}_v$ is sectorial), and
$$
(1 - \varepsilon)\mathcal{E}_1(f,f) \leq \mathcal{E}_{v, \alpha}(f,f) \leq \max\{ 1 + b_1, 1 + b + a - \varepsilon \}\mathcal{E}_1(f,f). \leqno{(3.4)}
$$ 
Moreover, if we denote by $(P_t^v)_{t \geq 0}$ the semigroup associated to $(\mathcal{E}_v, D(\mathcal{E}_v))$, then $(P^v_t)_{t \geq 0}$ is Markovian.

\vspace{0.2cm}

The following result covers and extends Theorem 3.6 from \cite{BoRoZh00}; see also \cite{Hi98} and \cite{Hi00}.

\begin{coro} \label{thm 3.10} i) Assume that $(\mathcal{E}, D(\mathcal{E}))$ 
satisfies inequality (3.2) for some strictly positive $\phi \in L^2(m)$ such that there exists $t_0 > 0$ with $P^v_{t_0}\phi \leq \phi$ and for some $r_0 < \dfrac{1 - \varepsilon}{a}$. 
Then there exists $\rho \in L^2(m)$ such that $\rho \cdot m$ is $(P^v_t)_{t \geq 0}$-invariant.

ii) If $(\mathcal{E}, D(\mathcal{E}))$ satisfies the $F$-Sobolev inequality such that $F^{-1} < \infty$ 
then the assumptions in i) are fulfilled for $\phi \equiv 1$.
\end{coro}

\begin{proof}
Since ii) follows by Remark 3.7, we prove only the first statement. 
It is straightforward to check that under i) and taking into account the first inequality in (3.4), there exists $\gamma(r)$ such that
$$
\mu(f^2) \leq r\mathcal{E}_{v, \alpha}(f,f) + \gamma(r)\mu(\phi|f|)^2, \; f \in D(\mathcal{E}), \; r > \widetilde{r_0},
$$
where $\widetilde{r_0} = \dfrac{r_0}{(1 + r_0)(1 - \varepsilon)}$. 
Since $r_0 < \dfrac{1 - \varepsilon}{a}$ and $\alpha = a + 1 - \varepsilon$, it follows that $\alpha < (\widetilde{r_0})^{-1}$, 
and by applying Lemma \ref{thm 3.8} we obtain the desired conclusion.
\end{proof}

\vspace{2mm}

\noindent
{\bf Acknowledgments.} Financial support by the Deutsche Forschungsgemeinschaft, project number RO1195/10-1 is gratefully
acknowledged.  We would also like to thank Martin Hairer for pointing out to us an 
error in an example contained in a previous version of this paper.
The first-named author acknowledges  support from the Romanian National
Authority for Scientific Research, CNCS-UEFISCDI, project number PN-II-IDPCE-2011-3-0045.
The second-named author is grateful to his father's support.

\section{Appendix}

\noindent
{\bf A.1 Proof of Lemma 2.1}. i) Let $f \in L^1_+(m)$ and define $P_t^*f : = \mathop{\sup}\limits_n P_t^* (f\wedge n)$ in $L^1(m)$.
By monotone convergence we get that $\|P_t^*f\|_1=m(P_t^*f)=\mathop{\sup}\limits_n m(f\wedge n)=m(f)<\infty$.
Moreover, $\|P_t^*f - f\|_1 \leq \|P_t^*(f\wedge n) - f \wedge n\|_1 + \|P_t^*f - P_t^*(f\wedge n)\|_1 + \|f - f\wedge n\|_1$ $\leq \|P_t^*(f\wedge n) - f \wedge n\|_{p'} + 2\|f - f\wedge n\|_1 \mathop{\longrightarrow}\limits_{t \to 0} 2\|f - f\wedge n\|_1 \mathop{\longrightarrow}\limits_{n \to \infty} 0$. 
By linearity, it follows that $(P_t^*)_{t\geq 0}$ is a strongly continuous semigroup on $L^1(m)$. 
The fact that $(P_t^*)_{t\geq 0}$ consists of positivity preserving operators is straightforward, by duality.

\vspace{0.2cm}

ii). If $(P_t)_{t\geq 0}$ satisfies (A$_2$) or (A$_1$) for $p=1$ then $(P_t)_{t\geq 0}$ may also be regarded as a semigroup of Markovian operators on $L^\infty(m)$. 
Denoting by $(P_t^*)_{t\geq 0}$ its adjoint on $(L^\infty(m))^\ast$, we prove first that $(P_t^*)_{t\geq 0}$ may be restricted to $L^1(m)$. 
To this end, let $f \in L^1_+(m)$. 
Then $P_t^*f$ is a positive finitely additive measure on $(E, \mathcal{B})$ which is absolutely continuous with respect to $m$. 
Let $(A_n)_n$ be a sequence of mutually disjoint $\mathcal{B}$-measurable sets. 
Then $P_t^*f(\mathop{\cup}\limits_n A_n)
=m(fP_t(1_{\mathop{\cup}\limits_n A_n}))=\mathop{\sum}\limits_n m(fP_t(1_{A_n})) = \mathop{\sum}\limits_n P_t^*f(A_n)$, 
hence $P_t^*f$ is a $\sigma$-additive positive measure, absolutely continuous with respect to $m$. 
By Radon-Nikodym it follows that $P_t^*f$ identifies with an element from $L^1_+(m)$. 
By linearity, we obtain that $(P_t^*)_{t\geq 0}$ may be regarded as a semigroup of positivity preserving operators on $L^1(m)$.

Now, since for each $t \geq 0$ the measure $p\mathcal{B} \ni f \mapsto m(\int_0^tP_sfds)$ 
is absolutely continuous with respect to $m$, there exists $\varphi_t \in L^1_+(m)$ such that $m(\int_0^t P_s f ds)=m(\varphi_t f)$, i.e. ii.1) holds. 
Then, $P_t^* \varphi_s (g) = m(\varphi_s P_t g) = 
m(\int_0^sP_{r+t}g dr) = m(\int_0^{s+t}P_rg dr) - m(\int_0^tP_{r}g dr) = \varphi_{t+s} - \varphi_s$, which proves ii.2).
Finally, $\| \dfrac{1}{s} (\varphi_{t + s} - \varphi_s) \|_1 \leq \dfrac{1}{s} (\|\varphi_{t + s}\|_1 + \|\varphi_s \|_1) = \dfrac{t}{s} \mathop{\longrightarrow}\limits_{s \to \infty} 0$, where the last inequality follows from ii.1).

\vspace{0.2cm}
\noindent
{\bf A.2 Biting lemma} \; (cf. \cite{BrCh80}) {\it Let $(\Omega, \mathcal{F}, m)$ 
be a finite positive measure space and let $(f_n)_{n \geq 1}$ 
be a bounded sequence in $L^1(m)$, i.e. $\sup\limits_n \int_E|f_n|dm < \infty$.
Then there exist  a function $f \in L^1(m)$, a subsequence $(f_{n_k})_{k \geq 1}$ 
and a decreasing sequence of measurable sets $(B_r)_{r \geq 1}$ with $\lim\limits_{r \to \infty} m(B_r) = 0$ such that
$$
f_{n_k} \mathop{\rightharpoonup}\limits_{k \to \infty} f \; \mbox{weakly in} \; L^1(E \setminus B_r, m)
$$
for every fixed $r \geq 1$.}

\noindent
{\bf A.3 Koml\'os lemma} \; (cf. \cite{Ko67}) {\it Let $(\Omega, \mathcal{F}, m)$ 
be a finite positive measure space and $(f_n)_{n \geq 1}$ be a bounded sequence in $L^1(m)$.
Then there exist a function $f \in L^1(m)$ and a subsequence $(f_{n_k})_{k \geq 1}$ such that
$$
\dfrac{1}{N}\sum\limits_{k=1}^{N}f_{n_k} \mathop{\longrightarrow}\limits_{N \to \infty} f \; \mbox{almost surely}. \leqno{(a.1)}
$$
Moreover, the subsequence $(f_{n_k})_{k \geq 1}$ can be chosen in such a way that its further subsequence will also satisfy (a.1). }

\end{document}